\documentclass[11pt]{article}

\usepackage{amsfonts}
\usepackage{xr}
\usepackage{graphicx,psfrag}
\usepackage{amsmath}
\usepackage{color,soul,textcomp}
\usepackage[active]{srcltx}
\usepackage{enumerate}

\def\ds{\displaystyle}
\def\2{C^{1,2}(\R\times\R^N)}
\def\to{\rightarrow}
\def\e{\varepsilon}

\def\R{\mathbb{R}}
\def\N{\mathbb{N}}
\def\tilde{\widetilde}
\def\.{\cdot}

\def\epsilon{\varepsilon}

\newlength{\textlarg}

\newcommand{\be}{\begin{equation}}
\newcommand{\ee}{\end{equation}}
\newcommand{\baa}{\begin{array}}
\newcommand{\eaa}{\end{array}}
\newcommand{\ba}{\begin{eqnarray}}
\newcommand{\ea}{\end{eqnarray}}

\newtheorem{theorem}{Theorem}[section]

\newtheorem{assumption}[theorem]{Assumption}

\newtheorem{claim}[theorem]{Claim}

\newtheorem{definition}[theorem]{Definition}

\newtheorem{proposition}[theorem]{Proposition}
\newtheorem{remark}[theorem]{Remark}

\newenvironment{proof}[1][Proof]{\noindent\textit{#1.} }{\hfill \rule{0.5em}{0.5em}}

\def\Fi#1{\begin{formula}{#1}}
\def\Ff{\end{formula}\noindent}

\begin{document}
\author{Thomas Giletti$^{\hbox{\small{ a}}}$, Fran{\c{c}}ois Hamel$^{\hbox{\small{ b}}}$\\
\\
\footnotesize{$^{\hbox{a }}$Universit\'e de Lorraine, Institut \'Elie Cartan de Lorraine, UMR 7502}\\
\footnotesize{54506 Vandoeuvre-l\`es-Nancy Cedex, France}\\
\footnotesize{$^{\hbox{b }}$Aix Marseille Universit\'e, CNRS, Centrale Marseille}\\
\footnotesize{Institut de Math\'ematiques de Marseille, UMR 7373, 13453 Marseille, France}}
\date{}
\title{\bf{Sharp thresholds between finite spread and uniform convergence for a reaction-diffusion equation with oscillating initial data}\thanks{This work has been carried out in the framework of Archim\`ede LabEx (ANR-11-LABX-0033) and of the A*MIDEX project (ANR-11-IDEX-0001-02), funded by the ``Investissements d'Avenir" French Government program managed by the French National Research Agency (ANR). The research leading to these results has received funding from the European Research Council under the European Union's Seventh Framework Programme (FP/2007-2013) / ERC Grant Agreement n.321186 - ReaDi - Reaction-Diffusion Equations, Propagation and Modelling, and from the ANR project NONLOCAL (ANR-14-CE25-0013).}}

\maketitle

\begin{abstract}
We investigate the large-time dynamics of solutions of multi-dimensional reaction-diffusion equations with ignition type nonlinearities. We consider solutions which are in some sense locally persistent at large time and initial data which asymptotically oscillate around the ignition threshold. We show that, as time goes to infinity, any solution either converges uniformly in space to a constant state, or spreads with a finite speed uniformly in all directions. Furthermore, the transition between these two behaviors is sharp with respect to the period vector of the asymptotic profile of the initial data. We also show the convergence to planar fronts when the initial data are asymptotically periodic in one direction.
\end{abstract}


\section{Introduction}

In this paper, we study the large-time behavior of solutions of the Cauchy problem for the reaction-diffusion equation in the multi-dimensional space $\R^N$:
\begin{equation}\label{eqn:RD}
\left\{
\begin{array}{l}
\partial_t u = \Delta u + f(u), \; t>0 \mbox{ and } x \in \R^N, \vspace{3pt}\\
u (t=0 ,x) = u_0 (x)  \in L^\infty (\R^N,[0,1]),
\end{array}
\right.
\end{equation}
for some classes of oscillating initial data that we will make more precise below, see Assumptions~\ref{ass:ini1} and~\ref{ass:ini2} below.

We assume throughout the paper that $f : [0,1] \to \R $ is Lipschitz-continuous and satisfies:
\begin{equation}\label{ass_f}
\left\{
\begin{array}{l}
\exists\,\theta \in (0,1), \ f(u) = 0 \ \mbox{ in } [0,\theta] \cup \{1\}, \ f(u) >0 \mbox{ in } (\theta,1), \vspace{3pt}\\
\exists\,\rho \in (0, 1 -\theta) , \ \mbox{the restrictions of } f \mbox{ to } [\theta,\theta+\rho] \mbox{ and to } [1-\rho,1] \mbox{ are } C^1, \vspace{3pt}\\
f \mbox{ is nondecreasing on }  [\theta, \theta +\rho]\mbox{ and }f'(1) < 0.
\end{array}
\right.
\end{equation}
Such nonlinearities are commonly refered to as combustion nonlinearities, due to the fact that they quite naturally arise in combustion models, when looking at the propagation of a flame. In such a framework, the quantity~$u$ represents the temperature, and the real number~$\theta$ is the ignition threshold: reaction occurs only when the temperature is higher than $\theta$. We note immediately that, due to the hypothesis~\eqref{ass_f} and the maximum principle, the solutions $u$ of~\eqref{eqn:RD} are well defined and classical in~$(0,+\infty)\times\R^N$, and they satisfy $0\leq u(t,x) \leq 1$ for all $t >0$ and $x \in \R^N$.

A considerable amount of work in the literature has been devoted to the study of propagation phenomena which arise in such equations, provided that $u_0$ is both localized and large enough to activate the reaction. A simple but meaningful way to investigate propagation dynamics is to look at planar traveling fronts connecting two stationary states, namely particular solutions of the type
$$u (t,x) = U (x \cdot e - ct),$$
where~$c$ is the speed of the front and $e$ belonging to the unit sphere $S^{N-1}$ is its direction, together with the following conditions:
$$\alpha < U (\cdot) < 1 \mbox{ in } \R, \; U (+\infty) = \alpha,\  U (-\infty) = 1,\ \alpha\in[0,\theta].$$
Finding a traveling front is clearly equivalent to solve
\begin{equation}\label{eqn:front_eq}
U'' + c U' + f(U) = 0 \ \mbox{ in } \R,\ \ U (+\infty) = \alpha < U (\cdot) <  U (-\infty)=1.
\end{equation}
It is well-known that, for any $\alpha \in [0,\theta)$, the solution $(c,U)$ exists and is unique, up to shifts of~$U$~\cite{AW,Fife} (note that it does not depend on the direction of propagation~$e$). From now on, we will refer to this unique traveling front as $(c^* (\alpha) ,U_\alpha)$. Furthermore, the speed $c^*(\alpha)$ is positive. The case $\alpha = \theta$ is very different, because it is an unstable equilibrium from above. Indeed, the restriction of $f$ to $(\theta,1)$ is positive and of the monostable type. The traveling front is then no longer unique, as there is instead some $c^* (\theta)>0$ and a continuum of admissible speeds $[c^* (\theta ) ,+\infty)$ for which traveling fronts exist~\cite{AW,KPP}. However, for a given speed $c \geq c^* (\theta)$, the profile $U$ of the front connecting $\theta$ to 1 is unique up to shifts. It is also known that the mapping $\alpha \in [0,\theta] \mapsto c^* (\alpha)$ we just defined is continuous and increasing, as for instance follows from the construction of traveling fronts in~\cite{AW,BN}.

Furthermore, those traveling fronts turn out to be attractive for the Cauchy problem~\eqref{eqn:RD}. On the one hand, given $\eta \in (0,1-\theta)$, if~$u_0$ lies above $\theta + \eta > \theta$ on a large enough set and $ u_0 (x) \to \alpha \in [0,\theta)$ as $\|x\| \to +\infty$,\footnote{Throughout the paper, we denote $\|x\|= \sqrt{x_1^2 + ... + x_N^2}$ the Euclidean norm of $x \in \R^N$.} then $u(t,x)$ spreads from $\alpha$ to 1 with the speed~$c^* (\alpha)$ uniformly in all directions (the precise meaning will be given in Definition~\ref{def:spread} below), as easily follows from~\cite{AW} and the maximum principle. It is even known, at least in one dimensional domains, that the profile of the solution also converges along the propagation to the front~$U_\alpha$~\cite{Kanel'1,Kanel'2}. On the other hand, because they are not unique, stability of monostable traveling fronts is slightly weaker, and only holds with respect to fastly decaying initial data. We refer the reader to~\cite{AW} for spreading speeds, \cite{bramson,HNRR,Kametaka,Lau85} for the profile's convergence in the so-called KPP case, and~\cite{sattinger,uchiyama} for the more general monostable case.

The goal of this work is to study spreading dynamics for some more complex initial data. Our first assumption is the following:

\begin{assumption}\label{ass:ini1}
The solution of~\eqref{eqn:RD} is such that
\begin{equation}\label{eqn:spread}
u(t,x) \to 1 \mbox{ locally uniformly in } x \in \R^N \mbox{ as } t \to +\infty.
\end{equation}
\end{assumption}

As explained above, Assumption~\ref{eqn:RD} is known to hold when $\eta \in (0, 1-\theta)$ is given and $u_0 \geq \theta +\eta$ on a large enough set, see~\cite{AW}. Together with the nonnegativity of $u_0$, this actually implies that the solution spreads at least with speed $c^* (0)$ in all directions. Another typical asymptotic behavior is the convergence as $t \to +\infty$ locally uniformly in~$x$ to a stationary state $\alpha \in [0,\theta]$.\footnote{Nevertheless, in general, the solutions $u$ may not converge to a steady state (explicit examples can be found in~\cite{py1,py2} with, say, $u_0\ge\theta$, $\lim_{\|x\|\to+\infty}u_0(x)=\theta$ and $f(s)$ of the type $(s-\theta)^p$ in a right neighborhood of $\theta$, for some large~$p$ and~$N$).} It is actually known, in dimension~$1$, that a sharp threshold can be observed between those two outcomes when one considers a monotone family of compactly supported~\cite{DM,Z} or $L^2$ \cite{MZ} initial data (see also~\cite{Polacik} for related results in higher dimension for bistable type nonlinearities). However, when convergence to some $\alpha \in [0,\theta]$ occurs, the solution does not exhibit front-like behavior, so that we only restrict ourselves to the spreading case and assume~\eqref{eqn:spread}.

While most of the literature dealt with initial data that converge to some stationary state $\alpha\in[0,\theta]$ far from the origin, it is also reasonable to look at a more general situation. The initial condition $u_0$ itself may indeed be given as the result of an earlier process and may therefore involve some strong heterogeneities at infinity. With this in mind, our main assumption will be that~$u_0$ behaves as a nontrivial periodic function as $\|x\| \to +\infty$, see Assumption~\ref{ass:ini2} below. To state it, we say that a function~$v_0 \in L^\infty ( \R^N,[0,1] )$ is $(1,...,1)$-periodic if it satisfies
$$\forall\, (k_1,...,k_N) \in \mathbb{Z}^N , \; \forall\, x=(x_1,...x_N) \in \R^N, \; v_0 (x_1 +k_1,...,x_N +k_N)= v_0 (x_1,...,x_N).$$
For any such function, we also define its average as
$$\overline{v_0} := \int_{[0,1]^N} v_0 (x)\, dx.$$
Our second assumption on $u_0$ then reads as follows:

\begin{assumption}\label{ass:ini2}
The initial condition $u_0\in L^{\infty}(\R^N,[0,1])$ converges to a periodic function as~$\|x\| \to +\infty$ in the following sense:
\begin{equation}\label{eqn:period}
\exists\,v_0 \  (1,...,1)\mbox{-periodic} , \; \exists\, L_1, ...,L_N >0 , \ \  u_0 (x) - v_{0} \Big(\frac{x_1}{L_1},...,\frac{x_N}{L_N}  \Big) \rightarrow 0 \mbox{ as } \|x\| \rightarrow +\infty.
\end{equation}
For convenience, we will denote by $L = (L_1,...,L_N)= \| L\|\,\xi_L$ the period vector and by $\xi_L =L/\|L\|$ the normalized period vector, which belongs to $S_+^{N-1}$, the part of the unit sphere with $($strictly$)$ positive coordinates.
\end{assumption}

As there exist traveling fronts $U_\alpha$ for problem~\eqref{eqn:front_eq} for any $\alpha \in [0,\theta)$ and a family of traveling fronts parametrized by $c\in[c^*(\theta),+\infty)$ when $\alpha=\theta$, it is not clear a priori what dynamics the initial data $u_0$ satisfying Assumptions~\ref{ass:ini1} and~\ref{ass:ini2} will produce. In~\cite{HS}, the case $\mathrm{ess } \sup_{\R} v_0 < \theta$ was investigated in dimension~1. Because reaction does not occur ahead of the propagation (where $u$ is below $\theta$), it was revealed that the solution spreads at the unique speed $c^* (\overline{v_0})$ of the traveling front connecting $\overline{v_0}$ to~$1$. On the other hand, it follows easily from the properties of~$f$ and the maximum principle that, if~$u_0$ satisfies~\eqref{eqn:period} in $\R^N$ and $\mathrm{ess} \inf_{\R^N} v_0 > \theta$, then the solution $u$ converges to~1 uniformly in $\R^N$ as $t \to +\infty$. Such observations can easily be extended to the cases~$\mathrm{ess} \sup_{\R^N} v_0 \leq \theta$ and $\mathrm{ess} \inf_{\R^N} v_0 \geq \theta$, provided that $v_0 \not \equiv \theta$ (by using the strong maximum principle and looking at~\eqref{eqn:RD} with initial data~$u(\varepsilon,\cdot)$ for some small $\varepsilon >0$). 

Therefore, our goal in this work is to deal with the case of a $(1,...,1)$-periodic function~$v_0$ oscillating around the ignition threshold~$\theta$:
\begin{equation}\label{eqn:cond_v0}
0 \leq \mathop{\mathrm{ess } \inf}_{\R^N} v_0 < \theta < \mathop{\mathrm{ess } \sup}_{\R^N} v_0 \leq 1.
\end{equation}
Unlike in~\cite{HS}, a non trivial interaction occurs ahead of the propagation between reaction and diffusion processes (at least in some finite time interval $(0,t_0)$ with $t_0>0$), which may trigger various large-time dynamics.

\begin{remark}{\rm
Note that Assumption~$\ref{ass:ini2}$ is clearly equivalent to the convergence of the initial condition~$u_0$ to a $(L_1,...,L_N)$-periodic function as $\| x \| \to +\infty$. However, we chose to state Assumption~$\ref{ass:ini2}$ by using~\eqref{eqn:period} in order to highlight in the following results the influence not only of the shape of the periodic function $v_0$, but also of the period vector $L$ and in particular its Euclidean norm~$\|L\|$.}
\end{remark}

Our first main theorem will deal, in any multi-dimensional space, with the spreading properties of~$u$. Let us define beforehand more rigorously the notion of spreading.

\begin{definition}\label{def:spread}
We say that a solution~$u$ of~\eqref{eqn:RD} spreads from $\alpha \in [0,\theta]$ to $1$ with speed~$c^*>0$ if
\begin{equation}
\left\{
\begin{array}{ll}
\forall \, 0 \leq  c < c^*, & \displaystyle \lim_{t \rightarrow +\infty} \sup_{\|x\| \leq ct} |u(t,x) -1| = 0,\vspace{3pt}\\
\forall \, c>c^*, & \displaystyle \lim_{t \rightarrow +\infty} \sup_{\|x\| \geq ct} |u (t,x) - \alpha | =0.
\end{array}
\right.
\end{equation}
\end{definition}

Because the large-time and large-space dynamics will a priori depend on the initial data, we do not impose in this definition what the invaded equilibrium state $\alpha$ is. In our first main result below, the existence of~$\alpha$ such that the solution spreads from~$\alpha$ to~1 will be guaranteed and its value will be expressed in terms of $v_0$ and the period vector $L$.

\begin{theorem}\label{th:spreading_ignition}
Let $v_0 : \R^N \to [0,1]$ be a given $(1,...,1)$-periodic function satisfying~\eqref{eqn:cond_v0} and let~$u_0\in L^{\infty}(\R^N,[0,1])$ satisfy Assumptions~$\ref{ass:ini1}$ and~$\ref{ass:ini2}$.\par
{\rm{(a)}} If $\overline{v_0} \geq \theta$, then the solution $u$ of~\eqref{eqn:RD} converges to~$1$ uniformly in $\R^N$ as $t \to +\infty$.\par
{\rm{(b)}} If $\overline{v_0} < \theta$, then there exists a continuous map 
$$\begin{array}{rcl}
S_+^{N-1} & \to &  (0,+\infty) \vspace{3pt} \\
\xi & \mapsto & L^* (\xi)
\end{array}$$
such that $\inf_{\xi \in S_+^{N-1}} L^* (\xi) >0$ and, for any $\xi \in S_+^{N-1}$, there is a continuous and increasing map
$$\begin{array}{rcl}
(0, L^* (\xi)] & \to &  (\overline{v_0},\theta] \vspace{3pt} \\
\lambda & \mapsto & \theta_{\xi,\lambda}
\end{array}$$
such that $\theta_{\xi,L^* (\xi)} = \theta$ and $\lim_{\lambda \to 0^+} \theta_{\xi,\lambda} =\overline{v_0} >0$. Moreover, the solution~$u$ of~\eqref{eqn:RD} satisfies the following spreading properties:
\begin{itemize}
\item[$(i)$] if $0 < \|L\| < L^* (\xi_L)$, then $u$ spreads from $\theta_{\xi_L,\|L\|} \in (\overline{v_0},\theta)$ to $1$ with speed $c^* (\theta_{\xi_L,\|L\|})>0$ in the sense of Definition~$\ref{def:spread}$;
\item[$(ii)$] if $\| L\| >L^* (\xi_L)$, then $u$ converges to $1$ uniformly in $\R^N$ as $t \to +\infty$.
\end{itemize}
\end{theorem}

The key point of the proof, when $\overline{v_0}<\theta$ and $\lambda>0$ is not too large, is that the real number $\theta_{\xi,\lambda}$ will be defined as the uniform limit of $v_{\xi,\lambda} (t,x)$ as $t \to +\infty$, where~$v_{\xi,\lambda}$ is the $(L_1,...,L_N)$-periodic solution of~\eqref{eqn:RD} with initial data $v_0 (x_1/L_1 ,... ,x_N /L_N)$ and $(L_1 ,..., L_N) = \lambda\,\xi$.  Since~$u$ will become close to $\theta_{\xi,\lambda}$ at large time and large $\|x\|$, we indeed expect that the large-time dynamics of~$u$ should be governed by the front connecting $\theta_{\xi,\lambda}$ to $1$, at least when~$\theta_{\xi,\lambda} < \theta$. In the proof, these real numbers~$\theta_{\xi,\lambda}$ will actually be defined for all $(\xi,\lambda)\in S^{N-1}_+\times(0,+\infty)$ and also if $\overline{v_0}\ge\theta$, and they will actually be equal to $1$ if~$\overline{v_0}\ge\theta$ or if $\overline{v_0}<\theta$ and $\lambda>0$ is large enough (see Theorem~\ref{th:periodic_pb} below). 

In some sense, Theorem~\ref{th:spreading_ignition} is in line with some previous works~\cite{DM,MZ,Polacik,Z}, as it further highlights the existence of sharp thresholds between various typical behaviors in the large-time dynamics of solutions of the reaction-diffusion equation~\eqref{eqn:RD}. In the aforementioned works, monotonically increasing families of compactly supported (or decaying to~$0$ at infinity) initial data were typically considered, leading to ordered families of solutions. In the present paper, the threshold-type result arises from the influence of the asymptotic period vector of the initial oscillations on the large-time dynamics of solutions. Here, when $\xi_L\in S^{N-1}_+$ is fixed and $\|L\|$ varies in $(0,+\infty)$, the initial conditions~$u_0$ satisfying~\eqref{eqn:period} and~\eqref{eqn:cond_v0} cannot be ordered. However, up to a scaling (see~\eqref{scaling} below), the periodic solutions of~\eqref{eqn:RD} with initial conditions $v_0(x_1/L_1,...,x_N/L_N)$ can be compared, due to the nonnegativity of the function~$f$. We also point out that Theorem~\ref{th:spreading_ignition} differentiates between spreading with finite speed or uniform convergence to~$1$, rather than between spreading or extinction (that is, uniform convergence to~$0$) as in~\cite{DM,MZ,Polacik,Z}.

It is important to note that, for a given periodic function~$v_0$ and a given normalized period vector~$\xi_L$, the spreading speed of the solution~$u$ is uniform in all directions when $\|L\| < L^* (\xi_L)$, and the convergence to 1 as $t\to +\infty$ is also uniform in $\R^N$ when $\| L\| > L^* (\xi_L)$. However, the critical magnitude $L^* (\xi)$ of the period vector depends on the normalized period vector~$\xi$, that is, in some sense, on the ratios between the components of the period vector. This is due to the lack of symmetry, in general, of~$v_0$. 

For instance, choose $N=2$ and $v_0\in L^{\infty}(\R^2,[0,1])$ such that, for any $(x_1,x_2) \in [0,1]\times[0,1]$,
\begin{equation}\label{defv0x12}
v_0 (x_1,x_2) = \frac{1 +\theta}{2} \chi_{[0,1/3 )} (x_1) + \frac{\theta}{2} \chi_{[1/3,1)} (x_1),\footnote{Throughout the paper, $\chi_E$ is the generic notation for the characteristic function of a set $E$.}
\end{equation}
with $\theta > 1/3$, so that $\overline{v_0} = (1+3\theta)/6< \theta$. Then one can check, by using homogenization techniques as~$\xi \to (0,1)$ or $(1,0)$, that $L^* (\xi)$ is not constant with respect to~$\xi \in S_+^{1}$ (see Proposition~\ref{prop:hom} below). More precisely, $L^* (\xi)$ is finite for all $\xi \in S_+^1$, while $L^* (\xi) \to +\infty$ as $\xi \to (0,1)$ with~$\xi \in S_+^{1}$ due to the fact that for any $x_2$ the average of~$v_0$ in the $x_1$-direction lies below the ignition threshold~$\theta$. Hence, under the assumptions of Theorem~\ref{th:spreading_ignition} for problem~\eqref{eqn:RD} in $\R^N$ with $\overline{v_0} < \theta$, one has~$\inf_{\xi \in S_+^{N-1}} L^* (\xi)>0$ while the quantity $\sup_{\xi \in S_+^{N-1}} L^* (\xi)$ is not finite in general. Coming back to~\eqref{defv0x12} with $N=2$, it immediately follows that, when $\lambda > \inf_{\xi \in S_+^1}  L^* (\xi)$, there exist $\xi_1$ and~$\xi_2$ in~$S_+^1$ such that $L^* (\xi_2) < \lambda < L^* (\xi_1)$ and thus~$u$ spreads from $\theta_{\xi_1 , \lambda}\in(0,\theta)$ to 1 with finite speed~$c^* (\theta_{\xi_1,\lambda})$ when $L=\lambda \xi_1 $, while $u$ converges uniformly to~1 as $t\to+\infty$ when $L = \lambda \xi_2$. Furthermore, with the same notations, $\overline{v_0} < \theta_{\xi_1,\lambda} < \theta$ and there exists $\lambda ' \in (0, L^* (\xi_2))$ such that $\theta_{\xi_2,\lambda'} = \theta_{\xi_1 ,\lambda}$, hence~$\theta_{\xi_2 ,\lambda'} \neq \theta_{\xi_1 ,\lambda'}$ since $0<\lambda'<L^*(\xi_2)<\lambda<L^*(\xi_1)$ and $\overline{v_0}<\theta_{\xi_1,\lambda'}<\theta_{\xi_1,\lambda}<\theta$. In other words, if $L = \lambda ' \xi_1$, then the solution $u$ spreads from $\theta_{\xi_1,\lambda'}\in(0,\theta)$ to~$1$ with the speed $c^* (\theta_{\xi_1,\lambda'})$ while if $L=\lambda ' \xi_2$, then the solution $u$ spreads from $\theta_{\xi_2 ,\lambda'}=\theta_{\xi_1,\lambda}\in(\theta_{\xi_1,\lambda'},\theta)$ to~$1$ with the larger speed $c^* (\theta_{\xi_2,\lambda'})=c^*(\theta_{\xi_1,\lambda})>c^* (\theta_{\xi_1,\lambda'})$, even though $\| \lambda ' \xi_1 \| = \| \lambda ' \xi_2 \| =\lambda'$.

\begin{remark}{\rm
Using the same arguments as in our paper, one may consider initial data $u_0$ that are asymptotically periodic, but not in the usual coordinates. More precisely let $v_0$ be a $(1,...,1)$-periodic function, choose some basis $(\mathrm{e}_1,..., \mathrm{e}_N)$ of $\R^N$ and some positive real numbers $L_1,...,L_N$. Then consider the solution of~\eqref{eqn:RD} with initial datum $u_0$ satisfying Assumption~$\ref{ass:ini1}$ and 
$$u_0 (y_1 \mathrm{e}_1 +... + y_N \mathrm{e}_N) - v_0\Big(\frac{y_1}{L_1},...,\frac{y_N }{L_N}\Big) \to 0 \ \mbox{ as } \|x\| \to +\infty ,$$
where $x= y_1 \mathrm{e}_1 + ... + y_N \mathrm{e}_N$ denotes the decomposition of $x$ in the $(\mathrm{e}_1, ... , \mathrm{e}_N)$ basis. As in the proofs of Theorem~\ref{th:spreading_ignition} and~\ref{th:periodic_pb} below, one can show that the solutions $w_{\xi,\lambda}$ of
\begin{equation*}
\left\{
\begin{array}{l}
\partial_t w_{\xi,\lambda} = \Delta_x w_{\xi,\lambda} + f(w_{\xi,\lambda}),\ \ t>0,\ x\in\R^N,\vspace{3pt}\\
w_{\xi,\lambda} (0,x) = \ds v_0\Big(\frac{y_1}{L_1},...,\frac{y_N }{L_N}\Big),
\end{array}
\right.
\end{equation*}
with $\lambda\,\xi = (L_1, ....,L_N)$ still converge uniformly in $\R^N$ as $t\to +\infty$ to some constants $\sigma_{\xi,\lambda} \in [0,\theta] \cup \{1\}$ which depend monotonically on $\lambda$. Under condition~\eqref{eqn:cond_v0}, for every $\xi\in S^{N-1}_+$, there still exists some~$\tilde{L}^* (\xi) \in [0,+\infty)$ such that $\sigma_{\xi,\lambda} = 1 $ if and only if $\lambda > \tilde{L}^* (\xi)$ (by stating $\tilde{L}^*(\xi)=0$ if $\overline{v_0}\ge\theta$). One may then conclude in a similar fashion as in Theorem~$\ref{th:spreading_ignition}$ that, depending on $\overline{v_0}$ and $L$ $($at least when~$\|L\|  \neq \tilde{L}^* (\xi_L)$$)$, the solution $u$ of~\eqref{eqn:RD} may either converge uniformly to $1$ or spread with a finite speed in all directions in the sense of Definition~$\ref{def:spread}$.}
\end{remark}

In the critical case $\|L\|=L^* (\xi_L)$, one cannot expect to get a similar result as Theorem~\ref{th:spreading_ignition}, where the dynamics depends only on $v_0$ and $L$. Indeed, as $\theta_{\xi_L, L^* (\xi_L)} = \theta$, the situation is asymptotically similar to the monostable equation, that is to the case when $u_0 \geq v_0 \equiv \theta$. Moreover, even if~$\theta_{\xi_L,L^* (\xi_L)} = \theta$, it follows from the proof of Theorem~\ref{th:spreading_ignition} that the solution $u$ of~\eqref{eqn:RD} with initial condition $u_0$ satis\-fying Assumptions~\ref{ass:ini1} and~\ref{ass:ini2}, where $v_0$ satisfies~\eqref{eqn:cond_v0}, is never larger than $\theta$ uniformly in space, that is~$\inf_{\R^N} u(t,\cdot) < \theta$ for all $t \geq 0$. Furthermore, independently of these considerations, the traveling fronts connecting~$\theta$ to~$1$ are no longer unique up to shifts, and there instead exists a continuum of admissible speeds $[c^* (\theta),+\infty)$. In such a case, in dimension~1, if $u_0 \geq \theta$ and $u_0 (+ \infty)= \theta$, the rate of convergence of~$u_0$ to~$\theta$ is known to play a crucial role in the selection of a spreading speed. Roughly speaking, the solution $u(t,x)$ behaves at large time as a front in the half line $\{ x \geq 0\}$ if the initial datum has a similar exponential decay as $x \to +\infty$ \cite{bramson,Kametaka,sattinger,uchiyama}. It is even possible to construct some initial data such that the spreading speed is infinite even though $u$ does not converge uniformly to 1 as $t\to +\infty$~\cite{HR}, or such that the spreading speed varies through the process of propagation~\cite{HN,yanagida}. In our framework, one may expect even more complex dynamics because of the oscillations around the state~$\theta$. Therefore, we only restrict ourselves to some particular cases. In the following statement, we denote
$$f' (\theta^+) := \lim_{s \to 0,\,s>0} \frac{f(\theta+s)}{s},$$
which from our assumption~\eqref{ass_f} is a well-defined nonnegative real number. Furthermore, it is also well known~\cite{AW,hadeler} that the minimal speed $c^*(\theta)$ of  traveling fronts $U:\R\to[\theta,1]$ connecting $\theta$ to $1$ satisfies~$c^*(\theta)\ge2\sqrt{f'(\theta^+)}$ (this inequality roughly means that the minimal front speed of a monostable reaction-diffusion equation is always larger than the minimal speed of the linearized problem around the invaded unstable state).

\begin{theorem}\label{th:critical_ignition}
Let $v_0 : \R^N \to [0,1]$ be a $(1,...,1)$-periodic function satisfying~\eqref{eqn:cond_v0} and $\overline{v_0} < \theta$. Choose some $\xi \in S_+^{N-1}$ and $L=L^* (\xi)\,\xi$, and an initial condition $u_0$ satisfying Assumptions~$\ref{ass:ini1}$ and~$\ref{ass:ini2}$. Assume also that
\begin{equation}\label{ass_u0v0}
u_0 (x) - v_0\Big(\frac{x_1}{L_1},...,\frac{x_N}{L_N}\Big) \leq A\,e^{-\lambda_0 \|x\|}\ \hbox{ for all }x\in\R^N,
\end{equation}
for some $A >0$ and $\lambda_0 $ such that
\be\label{deflambda*}
\lambda_0>0\ \hbox{ and }\ \lambda_0 \geq \lambda^* := \frac{c^* (\theta) - \sqrt{c^*(\theta)^2 -  4f'(\theta^+)}}{2},
\ee
where $\lambda^* \geq 0$ is the smallest root of $\lambda^2 - c^* (\theta) \lambda + f'(\theta^+) =0$.\par
Then the solution~$u$ of~\eqref{eqn:RD} spreads from $\theta$ to $1$ with speed $c^* (\theta) >0$ in the sense of Definition~$\ref{def:spread}$.
\end{theorem}

In the above theorem, the convergence rate of $u_0$ to $v_0$ is optimal. Indeed, in dimension~1 and in the limiting trivial case where $ v_0 \equiv \theta$ and $u_0= \theta + A\,e^{-\lambda |x|}$ for large enough~$x$ with $A>0$, $f'(\theta^+) >0$ and $0< \lambda < \lambda^*$, it is known that the solution~$u$ spreads with the speed $c=\lambda+f'(\theta^+)/\lambda > c^* (\theta)$~\cite{uchiyama}.

Note that the restriction of $f$ in the interval $[\theta,1]$ is of the monostable type, that is $f (\theta) = f(1)=0$ and $f >0$ in $(\theta,1)$. No further assumption on~$f$ is made in Theorem~\ref{th:critical_ignition}: in particular, $f_{| [\theta,1]}$ may or may not be of the more restricted classical KPP case, where it is typically assumed to be concave. Dynamics drastically differ between the KPP and the more general monostable case, as the propagation may be either pulled from ahead, or pushed by the whole front, see~\cite{gghr,stokes}. Our proof deals with both those situations and is partly inspired by the stability properties of traveling fronts with respect to fastly decaying perturbations~\cite{sattinger}.

Finally, we turn to a different situation where the initial condition is front-like and asymptotically periodic in a given direction~$e$. More precisely, we will make the following assumption:

\begin{assumption}\label{ass:front_like}
There exists a direction $e  \in S^{N-1}$ such that the initial condition $u_0$ satisfies the following two conditions:
\begin{equation*}
\left\{
\begin{array}{l}
\displaystyle \liminf_{ x \cdot e \to - \infty} u_0 (x) > \theta , \vspace{3pt}\\
\displaystyle  \exists\,v_0 \  (1,...,1)\mbox{-periodic} , \; \exists\,L_1, ...,L_N >0 , \quad u_0 (x) - v_{0} \Big(\frac{x_1}{L_1},..., \frac{x_N}{L_N}  \Big) \rightarrow 0 \mbox{ as } x \cdot e \rightarrow +\infty .
\end{array}\right.
\end{equation*}
Here the convergence is understood to hold uniformly with respect to the orthogonal directions.
\end{assumption}

This assumption means that, while the initial condition~$u_0$ still oscillates periodically around $\theta$ as~$x \cdot e \to +\infty$ (if $v_0$ still satisfies~\eqref{eqn:cond_v0}), it now lies above the ignition threshold $\theta$ in the opposite direction. Note that this latter fact immediately implies, by the classical results of~\cite{AW}, that Assumption~\ref{ass:ini1} holds. Furthermore, from the observations made in the first paragraph after Theorem~\ref{th:spreading_ignition}, the large-time dynamics of the solution~$u$ is expected to be still governed by the magnitude of the asymptotic period vector of $u_0$ as $x \cdot e \to + \infty$, in a similar fashion to Theorem~\ref{th:spreading_ignition}. Actually, with the same arguments as in the proof of Theorem~\ref{th:spreading_ignition}, it follows that if $\overline{v_0} < \theta$ and $u_0$ satisfies Assumption~\ref{ass:front_like} with $\|L\| < L^* (\xi_L)$, then $u$ spreads from $\theta_{\xi_L ,\|L\|}\in(\overline{v_0},\theta)$ to 1 in the direction~$e$ uniformly with respect to the orthogonal directions, in the sense that
\begin{equation*}
\left\{
\begin{array}{ll}
\forall \, 0 \leq  c < c^*, & \displaystyle \lim_{t \rightarrow +\infty} \sup_{ x \cdot e \leq ct} |u(t,x) -1| = 0,\vspace{3pt}\\
\forall \, c>c^*, & \displaystyle \lim_{t \rightarrow +\infty} \sup_{x \cdot e  \geq ct} |u (t,x) - \alpha | =0.
\end{array}
\right.
\end{equation*}
In fact, we will also prove in this framework a more precise result, namely that the profile of the solution converges to that of a traveling front. The main difficulty is that traveling fronts are very sensitive to perturbations which do not decay fast enough to the invaded state, so that the spatial oscillations of $u$ require a careful approach. Our last main theorem reads as follows.

\begin{theorem}\label{th:CV_profile}
Let $v_0 : \R^N \to [0,1]$ be a $(1,...,1)$-periodic function satisfying~\eqref{eqn:cond_v0} and $\overline{v_0} < \theta$, and let~$u_0$ be a front-like initial condition in the sense of Assumption~$\ref{ass:front_like}$. Assume moreover that~$0< \| L\| < L^* (\xi_L)$ and denote
$$\alpha = \theta_{\xi_L,\|L\|} \in (\overline{v_0},\theta).$$
Assume also that
\be\label{u0v0}
| u_0 (x) - v_0 (x)| \leq A\,e^{-\lambda_0\,x \cdot e}\ \hbox{ for all }x\in\R^N,
\ee
where $\lambda_0 >0$, $A>0$, and $e\in S^{N-1}$ is given by Assumption~$\ref{ass:front_like}$.\par
Then there exists a bounded function $m:[0,+\infty) \times ( \R e)^\perp \to \R$ such that
$$m(t,y) - m(t,0) \to 0 \mbox{ locally uniformly in } y \in (\R e )^\perp \ \mbox{ as } t \to +\infty$$
and
\be\label{convfront}
\big\| u \big(t, x + c^* (\alpha)\,t\,e + m(t,y)\,e\big) - U_{\alpha} (x \cdot e ) \big\|_{L^\infty (\R^N)} \to 0 \ \mbox{ as } t \to +\infty,
\ee
where $y = x - (x \cdot e) e$ is the component of $x$ which is orthogonal to~$e$, and $U_{\alpha}$ is the unique traveling front connecting $\alpha$ to $1$.
\end{theorem}

Let us point out here that $L^* (\xi)$ and $\theta_{\xi, \lambda}$ are the same mappings as in Theorem~\ref{th:spreading_ignition}, and only depend on the initial condition $u_0$ through the limiting periodic function $v_0$, as will be shown in Section~\ref{sec:periodic}.

\begin{remark}{\rm
As it will be clear from its proof, Theorem~$\ref{th:CV_profile}$ also holds when, instead of~\eqref{eqn:cond_v0}, we assume that
$$0 \leq \mathop{\mathrm{ess } \inf}_{\R^N} v_0 < \mathop{\mathrm{ess } \sup}_{\R^N} v_0 \leq \theta.$$
Namely, in such a case, one can set $L^* (\xi) = +\infty$ for all $\xi\in S^{N-1}_+$ and the conclusion of Theorem~$\ref{th:CV_profile}$ is fulfilled for all $L=(L_1,...,L_N)\in(0,+\infty)^N$ with $\theta_{\xi_L,\|L\|}=\overline{v_0}\in(0,\theta)$. Part (b)-(i) of Theorem~$\ref{th:spreading_ignition}$ also holds in this case for all $L=(L_1,...,L_N)\in(0,+\infty)^N$ with $\theta_{\xi_L,\|L\|}=\overline{v_0}\in(0,\theta)$ (we also recall that the one-dimensional case follows from~\cite{HS}).}
\end{remark}

\begin{remark}{\rm
In dimension~$1$, the methods we use in this paper allow us to deal with initial data that are asymptotically periodic to two different profiles as $x \to \pm \infty$. More precisely, let us fix two~$1$-periodic functions $v_0^\pm\in L^{\infty}(\R,[0,1])$ such that $\mathrm{ess} \inf_{\R} v_0^\pm < \theta < \mathrm{ess} \sup_{\R} v_0^\pm$, and two positive real numbers~$L^\pm$. Let $u$ be the solution of~\eqref{eqn:RD} with initial datum such that $u_0(x)- v_0^\pm (x/L^\pm)  \to 0$ as~$x \to \pm \infty$ and satisfying Assumption~$\ref{ass:ini1}$. Depending on the functions $v_0^\pm$ and on the values of~$L^\pm$, the solution $u$ either converges as $t\to+\infty$ to $1$ uniformly in $\R$, or converges to $1$ uniformly in~$\R_+$ or~$\R_-$ but spreads with a finite speed in the other direction, or spreads with finite speeds in both directions. More precisely, on the one hand, if $\overline{v_0^+} \geq \theta$ $($resp. $\overline{v_0^-} \geq \theta$$)$, then $u$ converges to $1$ uniformly in $\R_+$ $($resp. $\R_-$$)$. On the other hand, if $\overline{v_0^+} < \theta$ $($resp. $\overline{v_0^-} < \theta$$)$, then there exists $L^{*,+} >0$ $($resp. $L^{*,-}>0$$)$ such that if~$L^+ > L^{*,+}$ $($resp. $L^- > L^{*,-}$$)$, then $u$ converges uniformly to $1$ in $\R_+$ $($resp.~$\R_-$$)$, and if $L^+ < L^{*,+}$ $($resp. $L^- < L^{*,-}$$)$, then the solution spreads from some $\theta^+_{L^+}\in(\overline{v_0^+},\theta)$ $($resp. $\theta^-_{L^-}\in(\overline{v_0^-},\theta)$$)$ to $1$ with finite speed to the right $($resp. to the left$)$.}
\end{remark}

\paragraph{Outline.} In Section~2 the periodic problem~\eqref{eqn:RD} with initial condition $v_0 (x_1 / L_1 ,... ,x_N /L_N)$ is analyzed. After a suitable change of variable, we will define $L^* (\xi)$ and $\theta_{\xi,\lambda}$ and prove that these quantities satisfy the statements of Theorem~\ref{th:spreading_ignition}. Then, in Section~3, we will use the results on the periodic problem to get our spreading results, that is both Theorem~\ref{th:spreading_ignition} and~\ref{th:critical_ignition}. We will see that those results rely only on the asymptotic behavior of the solution of~\eqref{eqn:RD} with initial datum $v_0 (x_1 / L_1 ,... ,x_N /L_N)$. Special attention will be put on the construction of suitable supersolutions when the solution of the periodic problem converges to the critical value $\theta$ and when the initial condition satisfies the upper estimate~\eqref{ass_u0v0}. Section~4 is devoted to the proof of the convergence Theorem~\ref{th:CV_profile} for front-like solutions which are asymptotic oscillating in one direction. 


\section{The periodic reaction-diffusion equation}\label{sec:periodic}

Throughout this section, $v_0\in L^{\infty}(\R^N,[0,1])$ is a given $(1,...,1)$-periodic function satisfying~\eqref{eqn:cond_v0}. If~$u_0$ satisfies Assumption~\ref{ass:ini2}, then as $\|x\| \to +\infty$ the solution~$u$ shall be close to the solution $v_{\xi,\lambda}$ of~\eqref{eqn:RD} with initial data $v_0 (x_1/L_1 ,... ,x_N/L_N)$, where $\xi=\xi_L$ and $\lambda=\|L\|$, i.e. $L=(L_1,...,L_N)=\lambda\,\xi$. It is clear that for any positive time, $v_{\xi,\lambda}$ retains the $L$-periodicity of the initial data, so that it satisfies the following spatially periodic parabolic problem:
\begin{equation}\label{eqn:RD_per}
\left\{
\begin{array}{l}
\partial_t v_{\xi,\lambda}  = \Delta v_{\xi,\lambda} + f(v_{\xi,\lambda}), \ \ t>0 \mbox{ and } x \in\R^N , \vspace{3pt}\\
v_{\xi,\lambda} (t,\cdot) \mbox{ is } (L_1,...,L_N)\mbox{-periodic}, \vspace{3pt}\\
\displaystyle v_{\xi,\lambda} (0,x) = v_0 \left( \frac{x_1}{L_1} , ... , \frac{x_N}{L_N } \right).
\end{array}
\right.
\end{equation}
Notice also that $0 < v_{\xi,\lambda} (t,x) < 1$ for all $t >0$ and $x \in \R^N$ from the strong maximum principle. The first main step in the proof of Theorem~\ref{th:spreading_ignition} is to describe the asymptotic behavior of the solution $v_{\xi,\lambda}$ of~\eqref{eqn:RD_per} in terms of $\xi$ and $\lambda$. We prove the following:

\begin{theorem}\label{th:periodic_pb}
There exists a continuous map 
$$\left. \begin{array}{rcl}
S_+^{N-1} & \to &  [0,+\infty) \vspace{3pt} \\
\xi & \mapsto & L^* (\xi)
\end{array}
\right. $$
and, for any $\xi \in S_+^{N-1}$ with $L^*(\xi)>0$, there exists a continuous and increasing map
$$\left. \begin{array}{rcl}
(0, L^* (\xi)] & \to &  [0,\theta] \vspace{3pt} \\
\lambda & \mapsto & \theta_{\xi,\lambda},
\end{array}
\right. $$
such that:
\begin{itemize}
\item[$(i)$] if $\overline{v_0} \geq \theta$, then $L^*(\xi) = 0$ for every $\xi \in S_+^{N-1}$;
\item[$(ii)$] if $\overline{v_0} < \theta$, then $L^* (\xi) \in (0,+\infty)$, $\theta_{\xi,L^* (\xi)} = \theta$ and $\lim_{\lambda \to 0^+} \theta_{\xi,\lambda} =\overline{v_0} >0$ for every $\xi \in S_+^{N-1}$. Furthermore, $\inf_{\xi \in S_+^{N-1}} L^* (\xi) >0$.
\end{itemize}
Moreover:
\begin{enumerate}[(i)]
\item[$(iii)$] if $\overline{v_0}<\theta$ and $0 < \lambda \leq L^* (\xi)$, then $v_{\xi,\lambda}$ converges to $\theta_{\xi,\lambda}$ as $t \to +\infty$ uniformly in $\R^N$. Furthermore, the rate of convergence is exponential when $\theta_{\xi,\lambda} < \theta$, that is, when $\lambda < L^* (\xi)$;
\item[$(iv)$] if $\lambda >L^* (\xi)$ $($that is, if $\overline{v_0}\ge\theta$ with any $\lambda>0$, or if $\overline{v_0}<\theta$ with $\lambda>L^*(\xi)>0$$)$, then $v_{\xi,\lambda}$ converges to $1$ uniformly in $\R^N$ as $t \to +\infty$.
\end{enumerate}
\end{theorem}

\begin{remark}\label{rem:periodic_pb}{\rm
Note that one could easily extend this theorem while dropping~\eqref{eqn:cond_v0}, as other cases are actually trivial. On the one hand, if $\mathrm{ess} \sup_{\R^N} v_0 \leq \theta$, then $v_{\xi,\lambda}$ is a solution of the heat equation and then converges uniformly to $\overline{v_0}$ as $t\to+\infty$. On the other hand, if $\mathrm{ess} \inf_{\R^N} v_0 \geq \theta$ and $v_0 \not \equiv \theta$, then from the strong maximum principle and the positivity of $f$ on $(\theta,1)$, the solution $v_{\xi,\lambda}$ clearly converges uniformly to $1$ as $t\to+\infty$. Therefore, Theorem~$\ref{th:periodic_pb}$ completes the picture of large-time behaviors of solutions of the ignition type reaction-diffusion equation with periodic initial conditions.}
\end{remark}

\begin{proof}[Proof of Theorem~$\ref{th:periodic_pb}$] {\it Step 1: convergence of $v_{\xi,\lambda}$ to $\theta_{\xi,\lambda}\in[0,\theta]\cup\{1\}$.} Let us first fix $\lambda >0$ and~$\xi \in S_+^{N-1}$. The proof of the convergence of the solution $v_{\xi,\lambda}$ of~\eqref{eqn:RD_per} as $t \to +\infty$ to a constant stationary state is rather simple and relies mostly on the following classical Lyapunov function:
$$G[v (\cdot) ] = \int_{C_{\lambda \xi}} \Big(\frac{| \nabla v  (x) |^2}{2} - F(v  (x) )\Big)\, dx,$$
where
$$C_{\lambda \xi} := (0,L_1) \times ... \times (0, L_N)\ \hbox{ and }\ F(\varsigma) := \int_0^{\varsigma} f(s)ds.$$
From standard parabolic estimates and elementary arguments, the function $t\mapsto g(t):=G [ v_{\xi,\lambda} (t,\cdot) ]$ is differentiable in $(0,+\infty)$ with
\begin{equation*}
g'(t)=\frac{d G [v_{\xi,\lambda}  (t,\cdot)]}{dt}= - \int_{C_{\lambda \xi}} | \partial_t v_{\xi,\lambda}  (t,x) |^2 dx \leq 0,
\end{equation*}
and therefore this function is actually of class $C^1$ in $(0,+\infty)$ and of class $C^{1,\beta}$ in $[1,+\infty)$ for some~$\beta >0$. As $0\le v_{\xi,\lambda} \leq 1$, we know that $g(t)=G [v_{\xi,\lambda} (t,\cdot)]$ is bounded from below, independently of time, by $- L_1 \times ... \times L_N \times  \|  f \|_{L^\infty ([0,1])}$. Thus, $g(t)$ converges as $t \rightarrow +\infty$ to some constant and since~$g$ is of class~$C^{1,\beta}$ in $[1,+\infty)$, one infers that $g'(t)= -\|\partial_t v (t,\cdot) \|_{L^2 (C_{\lambda \xi})}  \rightarrow 0$ as~$t \rightarrow +\infty$. By standard parabolic estimates again, there is a sequence $(t_n)_{n\in\N}\to+\infty$ such that~$v_{\xi,\lambda} (t_n,\cdot)$ converges uniformly in $\R^N$ to an $L$-periodic classical stationary solution $\theta_{\xi,\lambda} (x)$ of~\eqref{eqn:RD} such that~$0 \leq \theta_{\xi,\lambda} (x) \leq 1$ for all $x \in \R^N$. Moreover, it is immediate to see that such stationary solutions are the constants in the set $[0, \theta] \cup \{1\}$: this can be checked by noting that $f(\theta_{\xi,\lambda} (x_0)) \leq 0$ for any $x_0$ such that~$\theta_{\xi,\lambda} (x_0) = \min_{\R^N} \theta_{\xi,\lambda}$. Thus, since~$f \geq 0$ in $[0,1]$ and $f>0$ in $(\theta,1)$, it follows from the strong maximum principle that
$$\theta_{\xi,\lambda} \equiv \theta_{\xi,\lambda} (x_0) \in [0,\theta] \cup \{1\}.$$

However, it remains to check that the limit $\theta_{\xi,\lambda}$ is unique and did not depend on the sequence~$(t_n)_{n\in\N}$. First note that $1$ is the only constant steady state with negative energy $G[1]<0$, so that the uniqueness of the limit is immediate if $\theta_{\xi,\lambda} = 1$. Assume then that $\theta_{\xi,\lambda} \leq \theta$. We proceed by contradiction and assume that $v_{\xi,\lambda} (t,\cdot)$ does not converge to $\theta_{\xi,\lambda}$ as $t \to +\infty$. By standard parabolic estimates, we extract another sequence $(t'_n)_{n\in\N}\to+\infty$ such that $v_{\xi,\lambda} (t'_n , \cdot)$ converges uniformly in $\R^N$ to some constant~$\theta' \neq \theta_{\xi,\lambda}$, which by the above argument is not equal to $1$ either. In particular, we can assume without loss of generality that $\theta_{\xi,\lambda} < \theta '$. Then for any $\varepsilon$ such that $0<\e<\theta'-\theta_{\xi,\lambda}\,(\le\theta-\theta_{\xi,\lambda})$, there exists some $n_0\in\N$ large enough so that
$$\| v_{\xi,\lambda} (t_{n_0},\cdot) - \theta_{\xi,\lambda} \|_{L^\infty (\R^N)} \leq \varepsilon.$$
It immediately follows from the comparison principle and the fact that $f \equiv 0$ in the interval $[\max \{ 0,\theta_{\xi,\lambda} - \varepsilon\}, \theta_{\xi,\lambda} + \varepsilon]$, that $\| v_{\xi ,\lambda} (t,\cdot) - \theta_{\xi,\lambda} \|_{L^\infty (\R^N)} \leq \varepsilon$ for all $t \geq t_{n_0}$. Hence, by passing to the limit as $t=t'_n\to+\infty$, one gets that $| \theta ' - \theta_{\xi,\lambda} | \leq \varepsilon$, which leads to a contradiction. We finally conclude that
$$v_{\xi,\lambda} (t,\cdot) \to \theta_{\xi,\lambda} \in [0,\theta] \cup \{1\} \mbox{ uniformly in $\R^N$ as } t \to +\infty.$$

{\it Step 2: definition of $L^*(\xi)$.} We now define the map
$$\xi \in S_+^{N-1} \mapsto L^* (\xi) := \inf \big\{ \lambda >0 \; | \ \theta_{\xi,\lambda} = 1 \big\} ,$$
and we shall prove that this map and the maps $\lambda \mapsto \theta_{\xi,\lambda}$ defined above for any $\xi \in S_+^{N-1}$ satisfy all the properties stated in Theorem~\ref{th:periodic_pb}. 

Let us first prove in this paragraph that $L^* (\xi)$ is a well-defined real number for each given~$\xi \in S_+^{N-1}$. Write $\xi=(\xi_1,...,\xi_N)$. Since $\theta < \mathrm{ess} \sup_{\R^N} v_0 \leq 1$, let us fix $\theta' \in (\theta, \mathrm{ess} \sup_{\R^N} v_0)$. From~\cite{AW}, there exists~$R >0$ such that any solution $u$ of~\eqref{eqn:RD} with $u_0 \geq \theta'$ in an Euclidean ball of radius $R$ converges locally uniformly to 1 as $t \to +\infty$. Let us now introduce the solution $w$ of the Cauchy problem
\begin{equation}\label{heatscaled}
\left\{
\begin{array}{l}
\partial_t w  =  \displaystyle \sum_{i=1}^N \xi_i^{-2} \partial_{x_i}^2 w , \ \ t>0 \mbox{ and } x \in \R^N, \vspace{3pt}\\
w (t,\cdot) \mbox{ is } (1,...,1)\mbox{-periodic}, \vspace{3pt}\\
\displaystyle w (0,x) = v_0 (x).
\end{array}
\right.
\end{equation}
Since $\max_{\R^N} w(t,\cdot) \to \mathrm{ess} \sup_{\R^N} v_0$ as $t \to 0^+$, there exists $t_0 >0$ such that $\max_{\R^N} w(t_0,\cdot) > \theta '$. By continuity of $w (t_0, \cdot)$, there exist $z_0 \in \R^N$ and $r >0$ such that $\min_{\overline{B_r (z_0)}} w (t_0,\cdot)\ge \theta'$, where $B_r (z_0)$ denotes the Euclidean ball of radius $r$ centered at $z_0$. For any $\lambda >0$ and $L = \lambda \xi$, the function 
\begin{equation}\label{defwxilambda}
w_{\xi, \lambda} (t,x) :=w \Big( \frac{t}{\lambda^2}, \frac{x_1}{L_1},... , \frac{x_N}{L_N} \Big)
\end{equation}
satisfies the Cauchy problem
\begin{equation}\label{eqwxilambda}
\left\{
\begin{array}{l}
\partial_t w_{\xi, \lambda}  =  \displaystyle \Delta w_{\xi, \lambda} , \ \ t>0 \mbox{ and } x \in \R^N, \vspace{3pt}\\
w_{\xi, \lambda} (t,\cdot) \mbox{ is } (L_1,...,L_N)\mbox{-periodic}, \vspace{3pt}\\
\displaystyle w_{\xi, \lambda} (0,x) = v_0 \left( \frac{x_1}{L_1},... , \frac{x_N}{L_N} \right) = v_{\xi,\lambda}(0,x).
\end{array}
\right.
\end{equation}
Therefore, there exists $\lambda_0 >0$ such that, for every $\lambda \geq \lambda_0$, there exists $x_\lambda \in \R^N$ with
$$w_{\xi,\lambda} (\lambda^2 t_0, \cdot) \geq \theta '\ \hbox{ in }B_R (x_\lambda).$$
Finally, the solution $v_{\xi,\lambda}$ of~\eqref{eqn:RD_per} satisfies $v_{\xi,\lambda} \geq w_{\xi,\lambda}$ in $[0,+\infty) \times \R^N$ from the parabolic maxi\-mum principle (since $f \geq 0$). Therefore, for every $\lambda \geq \lambda_0$, $v_{\xi,\lambda} (\lambda^2 t_0, \cdot) \geq \theta'$ in $B_R (x_\lambda)$, whence~$v_{\xi,\lambda} (\lambda^2 t_0 + t ,\cdot) \to 1$ locally uniformly in $\R^N$ as $t \to +\infty$ from the choice of $R>0$. By $L$-periodicity of $v_{\xi,\lambda}$, one concludes that, for every $\lambda \geq \lambda_0$, $v_{\xi,\lambda} (t,\cdot) \to 1$ as $t\to +\infty$ uniformly in $\R^N$, thus $\theta_{\xi,\lambda}=1$ and $L^* (\xi) \in \R$.

Then, the convergence of~$v_{\xi,\lambda}$ to~$\theta_{\xi,\lambda}$ in statement~$(iii)$ is immediate by construction: more precisely, as $t\to+\infty$, $v_{\xi,\lambda}(t,\cdot) \to \theta_{\xi,\lambda}\in[0,\theta]$ if $\lambda < L^* (\xi)$ while $v_{\xi,\lambda} \to \theta_{\xi,\lambda} \in [0,\theta] \cup \{1\}$ if~$\lambda = L^* (\xi)$ (provided that $L^*(\xi)>0$). But we still need to show that $\theta_{\xi,L^* (\xi)} =\theta$. Notice here that, when~$\theta_{\xi,\lambda} < \theta$, then there exists some time $T>0$ such that $0 \leq v_{\xi,\lambda} (T,\cdot) \leq \theta$ in $\R^N$. Thus, $v_{\xi,\lambda}$ is an $L$-periodic solution of the heat equation for $t \geq T$ and it follows that there exists $C >0$ such that, for all $t \geq T$ and $x\in \R^N$,
$$| v_{\xi,\lambda} (t,x) - \theta_{\xi,\lambda}|=\big | v_{\xi,\lambda} (t,x) -  \overline{v_{\xi,\lambda} (t,\cdot)} \big | \leq C e^{-\mu_1 (t - T)},$$
where the average $\overline{v_{\xi,\lambda} (t,\cdot)}$ over $C_{\lambda \xi}$ is necessarily equal to $\theta_{\xi,\lambda}$ for all $t \geq T$, and $\mu_1 >0$ is the second eigenvalue of the Laplacian in the subspace of $L$-periodic functions.

\smallbreak
{\it Step 3: monotonicity of $\theta_{\xi,\lambda}$ with respect to $\lambda>0$.} In order to show the other properties of Theorem~\ref{th:periodic_pb} and in particular the monotonicity of $\theta_{\xi,\lambda}$ with respect to $\lambda$, let us now make a change of variable similar to~\eqref{defwxilambda} and define the functions
\be\label{scaling}
\tilde{v}_{\xi,\lambda} (t,x_1,...,x_N) := v_{\xi,\lambda} (\lambda^2 t ,L_1 x_1, ..., L_N x_N).
\ee
They solve the parabolic Cauchy problems
\begin{equation}\label{eqn:RD_per2}
\left\{
\begin{array}{l}
\partial_t \tilde{v}_{\xi,\lambda}  =  \displaystyle \sum_{i=1}^N \xi_i^{-2}\partial_{x_i}^2 \tilde{v}_{\xi,\lambda} + \lambda^2 f(\tilde{v}_{\xi,\lambda}), \ \ t>0 \mbox{ and } x \in\R^N, \vspace{3pt}\\
\tilde{v}_{\xi,\lambda} (t,\cdot) \mbox{ is } (1,...,1)\mbox{-periodic}, \vspace{3pt}\\
\displaystyle \tilde{v}_{\xi,\lambda} (0,x) = v_0 (x),
\end{array}
\right.
\end{equation}
where, for each $1 \leq i \leq N$, $\xi_i=L_i/\lambda>0$. Moreover, since $f \geq 0$, it follows from the maximum principle that $(0,+\infty) \ni \lambda \mapsto \tilde{v}_{\xi,\lambda} (t,x)$ is nondecreasing for any $(t,x)\in[0,+\infty)\times\R^N$. In particular,
$$\theta_{\xi,\lambda} = \lim_{t \rightarrow +\infty} v_{\xi,\lambda} (t, \cdot)=\lim_{t \rightarrow +\infty} \tilde{v}_{\xi,\lambda} (t/\lambda^2 , \cdot)$$
is nondecreasing with respect to~$\lambda>0$. It follows from the definition of $L^*(\xi)$ that, for any $\xi\in S^{N-1}_+$ and $\lambda > L^* (\xi)$, the function $v_{\xi,\lambda}$ converges uniformly to $\theta_{\xi,\lambda}=1$ as~$t \to +\infty$ (this is statement $(iv)$ in Theorem~\ref{th:periodic_pb}, provided that one proves that $L^*(\xi)=0$ if $\overline{v_0}\ge\theta$).

Let us now check that $\lambda\mapsto\theta_{\xi,\lambda}$ is not only nondecreasing in $(0,+\infty)$, but also increasing in the interval $(0,L^* (\xi)]$, for any $\xi \in S_+^{N-1}$ such that $L^*(\xi)>0$. In this paragraph, we fix a normalized period vector $\xi\in S^{N-1}_+$ such that $L^*(\xi)>0$. From the previous paragraph, it is enough to show that~$\lambda \mapsto \theta_{\xi,\lambda}$ is increasing in the open interval $(0,L^* (\xi))$. Let us choose $0<\lambda_1 < \lambda_2<L^*(\xi)$ and let us prove that~$\theta_{\xi,\lambda_1}<\theta_{\xi, \lambda_2}$. As aforementioned, the maximum principle implies that
$$\tilde{v}_{\xi,\lambda_1} \leq \tilde{v}_{\xi,\lambda_2}\ \hbox{ in }[0,+\infty) \times \R^N.$$
We will actually show that the inequality is strict in $(0,+\infty) \times \R^N$. Assume by contradiction that there exists $(t_0,x_0) \in (0,+\infty) \times \R^N$ such that $\tilde{v}_{\xi,\lambda_1} (t_0,x_0)= \tilde{v}_{\xi,\lambda_2} (t_0,x_0)$. Since $\tilde{v}_{\xi,\lambda_1}$ is a subsolution of the equation~\eqref{eqn:RD_per2} satisfied by $\tilde{v}_{\xi,\lambda_2}$ with $\lambda=\lambda_2$, the strong maximum principle implies that $\tilde{v}_{\xi,\lambda_1} \equiv \tilde{v}_{\xi,\lambda_2}$ in $[0,t_0]\times \R^N$, whence $\lambda_1^2 f (\tilde{v}_{\xi,\lambda_1}) \equiv \lambda_2^2 f (\tilde{v}_{\xi,\lambda_1})$ and $f (\tilde{v}_{\xi,\lambda_1}) \equiv 0$ in $[0,t_0] \times \R^N$ (since $0<\lambda_1<\lambda_2$). On the other hand, by~\eqref{ass_f},~\eqref{eqn:cond_v0} and the periodicity of $\tilde{v}_{\xi,\lambda_1}$, there holds $\max_{\R^N} \tilde{v}_{\xi,\lambda_1} (t,\cdot) < 1$ for every $t >0$ and $\max_{\R^N} \tilde{v}_{\xi,\lambda_1} (t,\cdot) \to \mathrm{ess} \sup_{\R^N} v_0 > \theta$ as $ t \to 0^+$. Thus, there exists $t_1 \in (0,t_0)$ such that
$$\theta < \max_{\R^N} \tilde{v}_{\xi,\lambda_1} (t_1,\cdot) < 1.$$
This contradicts~\eqref{ass_f} and the fact $f (\tilde{v}_{\xi,\lambda_1} (t_1,\cdot)) \equiv 0$ in $\R^N$. Finally, $\tilde{v}_{\xi,\lambda_1}(t,x)<\tilde{v}_{\xi,\lambda_2}(t,x)$ for all~$(t,x)\in(0,+\infty)\times\R^N$ and, by periodicity and continuity,
$$\max_{\R^N} \big(\tilde{v}_{\xi,\lambda_1} (t,\cdot) - \tilde{v}_{\xi,\lambda_2} (t ,\cdot)\big)<0\ \hbox{ for all }t >0.$$
Moreover, one has
\begin{eqnarray*}
0 & = & \partial_t (\tilde{v}_{\xi,\lambda_1} - \tilde{v}_{\xi,\lambda_2}) - A \left[ \tilde{v}_{\xi,\lambda_1} - \tilde{v}_{\xi,\lambda_2}\right] - \lambda_1^2 f (\tilde{v}_{\xi,\lambda_1}) + \lambda_2^2 f (\tilde{v}_{\xi,\lambda_2})\vspace{3pt}\\
&= &  \partial_t (\tilde{v}_{\xi,\lambda_1} - \tilde{v}_{\xi,\lambda_2}) - A \left[\tilde{v}_{\xi,\lambda_1} - \tilde{v}_{\xi,\lambda_2}\right] - \lambda_2^2  \left( f (\tilde{v}_{\xi,\lambda_1}) - f (\tilde{v}_{\xi,\lambda_2})\right) + (\lambda_2^2 - \lambda_1^2) f (\tilde{v}_{\xi,\lambda_1})
\end{eqnarray*}
in $(0,+\infty)\times\R^N$, where $A[v] := \sum_{i=1}^N \xi_i^{-2} \partial_{x_i}^2 v$. Hence, since $f \geq 0$ and $0<\lambda_1 < \lambda_2 $, the function~$\tilde{v}_{\xi,\lambda_1} - \tilde{v}_{\xi,\lambda_2}$ is a subsolution of
\begin{equation}\label{eqn:supersol_1}
\partial_t (\tilde{v}_{\xi,\lambda_1} - \tilde{v}_{\xi,\lambda_2}) - A \left[ \tilde{v}_{\xi,\lambda_1} - \tilde{v}_{\xi,\lambda_2} \right] -  c \times (\tilde{v}_{\xi,\lambda_1} - \tilde{v}_{\xi,\lambda_2}) \leq 0
\end{equation}
in $(0,+\infty)\times\R^N$, where
$$c(t,x):=\lambda_2^2 \frac{ f (\tilde{v}_{\xi,\lambda_1} (t,x)) - f (\tilde{v}_{\xi,\lambda_2}(t,x))}{\tilde{v}_{\xi,\lambda_1}(t,x)- \tilde{v}_{\xi,\lambda_2} (t,x)}$$
for all $t>0$ and $x\in\R^N$ (remember that $\tilde{v}_{\xi,\lambda_1}(t,x)-\tilde{v}_{\xi,\lambda_2}(t,x)<0$ for all $t >0$ and $x\in \R^N$). Besides, since $\lambda_1$ and $\lambda_2$ are both (strictly) smaller than $L^* (\xi)$, we know that $\tilde{v}_{\xi,\lambda_1}$ and $\tilde{v}_{\xi,\lambda_2}$ converge uniformly in $\R^N$ to the constants $\theta_{\xi,\lambda_1} \leq \theta_{\xi,\lambda_2}$ in $[0,\theta]$ (they cannot be equal to 1 by definition of~$L^* (\xi)$). Since, by~\eqref{ass_f}, the function $f$ is nondecreasing on some interval $[0,\theta + \rho]$ with $\rho >0$, there exists $t_1 >0$ such that
$$c(t,x)=\lambda_2^2 \frac{ f (\tilde{v}_{\xi,\lambda_1} (t,x)) - f (\tilde{v}_{\xi,\lambda_2} (t,x))}{\tilde{v}_{\xi,\lambda_1} (t,x) - \tilde{v}_{\xi,\lambda_2} (t,x)} \geq 0\ \hbox{ for all }t \geq t_1\hbox{ and }x \in \R^N.$$
Therefore, the negative constant $\sigma := \max_{\R^N}\big(\tilde{v}_{\xi,\lambda_1} (t_1,\cdot) - \tilde{v}_{\xi,\lambda_2} (t_1,\cdot)\big)<0$ is a supersolution of equation~\eqref{eqn:supersol_1} in $[t_1,+\infty)\times\R^N$, that is, $-c(t,x)\,\sigma\ge0$ for all $(t,x)\in[t_1,+\infty)\times\R^N$. Thus,
$$\tilde{v}_{\xi,\lambda_1} (t,x) - \tilde{v}_{\xi,\lambda_2} (t,x) \leq \sigma<0\ \hbox{ for all }t \geq t_1\hbox{ and }x \in \R^N,$$
whence $\theta_{\xi,\lambda_1} - \theta_{\xi,\lambda_2}\le\sigma<0$. The strict monotonicity of $\lambda\mapsto\theta_{\xi,\lambda}$ in $(0, L^* (\xi))$, and hence in $(0,L^* (\xi)]$, is proved.

\smallbreak
{\it Step 4: $L^*(\xi)>0$ if $\overline{v_0}<\theta$ and $L^*(\xi)=0$ if $\overline{v_0}\ge\theta$.} As $\lambda \rightarrow 0^+$, since the right-hand side $\lambda^2 f(\tilde{v}_{\xi,\lambda})$ of~\eqref{eqn:RD_per2} converges to $0$ uniformly in $(0,+\infty)\times\R^N$, one infers that the function $\tilde{v}_{\xi,\lambda}$ converges locally uniformly in time in $[0,+\infty)$ and uniformly in space (by periodicity) to the solution $w$ of the scaled heat equation~\eqref{heatscaled} in $(0,+\infty)\times\R^N$ with initial condition $v_0$. Since $\|w(t,\cdot)-\overline{v_0}\|_{L^{\infty}(\R^N)}\to0$ as~$t\to+\infty$, it then follows that, for any $\varepsilon >0$, there exists $t_\e>0$ such that
$$\|v_{\xi,\lambda} (t_\e, \cdot)- \overline{v_0}\|_{L^{\infty}(\R^N)} \leq \varepsilon\ \hbox{ for all }\lambda>0\hbox{ small enough}.$$
On the one hand, when $\overline{v_0} < \theta$ and $\varepsilon>0$ is chosen so that $0\le\overline{v_0}-\epsilon<\overline{v_0}+\epsilon\le\theta$, both constants $\overline{v_0} \pm \varepsilon$ are solutions of~\eqref{eqn:RD}, and we get from the parabolic comparison principle that $\theta_{\xi,\lambda} \in [ \overline{v_0} - \varepsilon, \overline{v_0} + \varepsilon]$ for all $\lambda>0$ small enough. Hence,
$$\lim_{\lambda \rightarrow 0^+} \theta_{\xi,\lambda} = \overline{v_0}\ \hbox{ and }\ L^* (\xi) >0\ \hbox{ if }\overline{v_0}<\theta.$$
On the other hand, if $\overline{v_0} \geq \theta$, we get by the same reasoning that $\lim_{\lambda \to 0^+} \theta_{\xi,\lambda} \geq \theta$. In that case, if~$L^*(\xi)$ were (strictly) positive, it would then easily follow from the (strict) monotonicity of $\theta_{\xi,\lambda}$ with respect to $\lambda\in(0,L^*(\xi)]$ that $\theta_{\xi,\lambda} = 1 $ for all $\lambda>0$, which would contradict the definition of $L^*(\xi)$ and its assumed positivity. Therefore, if $\overline{v_0}\ge\theta$, then $L^*(\xi)=0$ and $\theta_{\xi,\lambda}=1$ for all $\lambda>0$.

\smallbreak
{\it Step 5: continuity properties and $\theta_{\xi,L^*(\xi)}=\theta$ if $\overline{v_0}<\theta$.} We have now proved statements $(i)$ to $(iv)$ of Theorem~\ref{th:periodic_pb}, apart from the fact that $\theta_{\xi,\lambda} = \theta$ when $\lambda= L^* (\xi) >0$ and the property $\inf_{\xi\in S^{N-1}_+}L^*(\xi)>0$ if $\overline{v_0}<\theta$. These facts will be proved below along with the continuity of $L^* (\xi)$ and~$\theta_{\xi,\lambda}$ with respect to $\xi \in S_+^{N-1}$ and $\lambda \in (0,L^* (\xi)]$ in case $\overline{v_0}<\theta$. In the remaining part of the proof, we can assume that $\overline{v_0} < \theta$ (since there is nothing more to prove in the other case).

We begin by partitioning $S_+^{N-1} \times(0,+\infty)$ into three sets as follows:
\begin{eqnarray*}
D_1 &:= & \big\{ (\xi,\lambda) \in S_+^{N-1} \times(0,+\infty) \; | \ \theta_{\xi,\lambda} =1 \big\},\\
D_2 &:= & \big\{ (\xi,\lambda) \in S_+^{N-1} \times(0,+\infty) \; | \ \theta_{\xi,\lambda} < \theta \big\},\\
D_3 & :=  & \big\{ (\xi,\lambda) \in S_+^{N-1} \times(0,+\infty) \; | \ \theta_{\xi,\lambda} = \theta \big\}.
\end{eqnarray*}
Firstly, let us check that~$D_1$ is open. Fix $(\xi,\lambda) \in D_1$. In other words, $\tilde{v}_{\xi,\lambda} (t,\cdot)\to1$ as $t \to +\infty$ uniformly in $\R^N$. In particular, there is $t_1>0$ such that $\min_{\R^N}\tilde{v}_{\xi,\lambda}(t_1,\cdot)>(1+\theta)/2$. But, from standard parabolic estimates, $\tilde{v}_{\xi',\lambda'}(t_1,\cdot)\to\tilde{v}_{\xi,\lambda}(t_1,\cdot)$ as $(\xi',\lambda')\to(\xi,\lambda)$ in $S^{N-1}_+\times(0,+\infty)$ locally uniformly in $\R^N$ (and then uniformly in $\R^N$ by $(1,...,1)$-periodicity). Therefore, there is $\eta>0$ such that
$$\min_{\R^N}\tilde{v}_{\xi',\lambda'}(t_1,\cdot)>\frac{1+\theta}{2}$$
for all $(\xi',\lambda')\in S^{N-1}_+\times(0,+\infty)$ with $\|\xi'-\xi\|+|\lambda'-\lambda|\le\eta$. For any such $(\xi',\lambda')$, the maximum principle and the assumption~\eqref{ass_f} imply that $\tilde{v}_{\xi',\lambda'}(t,\cdot)\ge\zeta(t-t_1)$ in $\R^N$ for all $t\ge t_1$, where the solution $\zeta$ of $\zeta'=f(\zeta)$ in $\R$ with $\zeta(0)=(1+\theta)/2$ satisfies $\zeta(+\infty)=1$. Since $\tilde{v}_{\xi',\lambda'}\le1$, one concludes that $\tilde{v}_{\xi',\lambda'}(t,\cdot)\to1$ as $t\to+\infty$ uniformly in $\R^N$, whence $\theta_{\xi',\lambda'}=1$. Finally, $D_1$ is open. From the definition of $L^*(\xi)$ and the fact that $\lambda\mapsto\theta_{\xi,\lambda}$ is nondecreasing in $(0,+\infty)$ for every $\xi\in S^{N-1}_+$, one also infers that the map $S_+^{N-1}\ni\xi \mapsto L^* (\xi)$ is upper semicontinuous.

Secondly, let us show that the set $D_2$ is open. Fix $(\xi,\lambda) \in D_2$. In this case, there exist some $\varepsilon >0$ and $t_1 >0$ such that 
$$\tilde{v}_{\xi,\lambda}  (t_1,\cdot) \leq \theta - \varepsilon\ \hbox{ in }\R^N.$$
Therefore, for any $(\xi',\lambda')\in S^{N-1}_+\times(0,+\infty)$ close enough to $(\xi,\lambda)$, the function $\tilde{v}_{\xi',\lambda'}$ satisfies~$\tilde{v}_{\xi',\lambda'} (t_1,\cdot) \leq \theta -\varepsilon/2$ in $\R^N$, whence $\theta_{\xi',\lambda'}\le\theta-\epsilon/2 < \theta$ and $(\xi',\lambda') \in D_2$. Thus, the set~$D_2$ is open. It also follows from the definition of $L^*(\xi)$, from the fact that $\lambda\mapsto\theta_{\xi,\lambda}$ is non\-decreasing in~$(0,+\infty)$ and smaller than $\theta$ in $(0,L^*(\xi))$, that $S_+^{N-1}\ni\xi \mapsto L^* (\xi)$ is lower semicontinuous. Henceforth, the map~$\xi\mapsto L^*(\xi)$ is continuous in $S^{N-1}_+$.

Since both $D_1$ and $D_2$ are open (and are not empty from the previous paragraphs), it follows that the set $D_3$ is not empty either. Moreover, from the strict monotonicity of $\theta_{\xi,\lambda}$ with respect to~$\lambda\in(0,L^*(\xi)]$, we infer that for each $\xi \in S_+^{N-1}$ there exists at most one $\lambda >0$ such that $\theta_{\xi,\lambda} = \theta$ and~$(\xi,\lambda) \in D_3$. By definition of $L^* (\xi)$ and the openness of $D_1$ and $D_2$, it is then clear that 
$$D_3 = \big\{ (\xi,\lambda)\in S^{N-1}_+\times(0,+\infty) \, | \ \lambda = L^* (\xi)\big\},$$
which ends the proof of part $(ii)$ of Theorem~\ref{th:periodic_pb}, apart from the property $\inf_{\xi\in S^{N-1}_+}L^*(\xi)>0$.

Now, let us show that the map $\lambda\mapsto\theta_{\xi,\lambda}$ is continuous in $(0,L^*(\xi)]$ for every $\xi\in S^{N-1}_+$. We will actually show a slightly more general continuity property. We choose $(\xi,\lambda) \in D_2 \cup D_3$ and a sequence~$(\xi_n,\lambda_n)_{n\in\N}$ in $D_2\cup D_3$ converging to $(\xi,\lambda)$ as $n\to+\infty$. For any $\varepsilon >0$, there exists $t_\e>0$ such that
$$\theta_{\xi,\lambda} - \frac{\varepsilon}{2} \leq \tilde{v}_{\xi,\lambda} (t_\e,\cdot) \leq \theta_{\xi,\lambda} + \frac{\varepsilon}{2}\ \hbox{ in }\R^N.$$
Thus, from the uniform convergence of $\tilde{v}_{\xi_n,\lambda_n}(t_\e,\cdot)$ to $\tilde{v}_{\xi,\lambda}(t_\e,\cdot)$ as $n\to+\infty$, we get that $\theta_{\xi,\lambda} - \varepsilon \leq \tilde{v}_{\xi_n,\lambda_n} (t_\e, \cdot) \leq \theta_{\xi,\lambda} + \varepsilon$ in $\R^N$ for all $n$ large enough. Hence, for any $\varepsilon>0$, there holds 
$$\max\{0,\theta_{\xi,\lambda} - \varepsilon\} \leq \theta_{\xi_n,\lambda_n} \leq \min \{ \theta, \theta_{\xi,\lambda} +\varepsilon \}$$
for all $n$ large enough. Therefore, the restriction of the map $(\xi,\lambda) \mapsto \theta_{\xi,\lambda}$ to $D_2 \cup D_3$ is continuous on that set. This yields in particular the continuity of the function $\lambda\mapsto\theta_{\xi,\lambda}$ in $(0,L^*(\xi)]$, for every~$\xi\in S^{N-1}_+$.

\smallbreak
{\it Step 6: $\inf_{\xi\in S^{N-1}_+}L^*(\xi)>0$ if $\overline{v_0}<\theta$.} Finally, in order to complete the proof of Theorem~\ref{th:periodic_pb}, let us show that $\inf_{\xi\in S^{N-1}_+}L^*(\xi)>0$, still assuming that $\overline{v_0}<\theta$. Let us first fix some real numbers~$\alpha$,~$\beta$ and $t_0>0$ such that $\overline{v_0}<\alpha<\beta<\theta$ and $\alpha+\|f\|_{L^{\infty}([0,1])}t_0\le\beta$. Remember that, for any~$(\xi,\lambda)\in S^{N-1}_+\times(0,+\infty)$ and $L=\lambda\,\xi$, $w_{\xi,\lambda}$ denotes the solution of the heat equation~\eqref{eqwxilambda} with initial condition $w_{\xi,\lambda}(0,x)=v_0(x_1/L_1,...,x_N/L_N)$. Since $\overline{v_0}<\alpha$, it follows from standard elementary arguments that there is $\lambda_0>0$ such that, for every $(\xi,\lambda)\in S^{N-1}_+\times(0,\lambda_0)$ and $L=\lambda\,\xi$, there holds
$$w_{\xi,\lambda}(t_0,\cdot)\le\alpha\ \hbox{ in }\R^N.$$
Moreover, for any $(\xi,\lambda)\in S^{N-1}_+\times(0,+\infty)$, the maximum principle implies that the solution $v_{\xi,\lambda}$ of~\eqref{eqn:RD_per} satisfies $0\le w_{\xi,\lambda}(t_0,\cdot)\le v_{\xi,\lambda}(t_0,\cdot)\le w_{\xi,\lambda}(t_0,\cdot)+\|f\|_{L^{\infty}([0,1])}t_0\ \hbox{ in }\R^N$. Therefore, for every~$(\xi,\lambda)\in S^{N-1}_+\times(0,\lambda_0)$ and $L=\lambda\,\xi$, one has
$$0\le v_{\xi,\lambda}(t_0,\cdot)\le\alpha+\|f\|_{L^{\infty}([0,1])}t_0\le\beta<\theta\ \hbox{ in }\R^N,$$
whence $v_{\xi,\lambda}$ solves the heat equation in $[t_0,+\infty)\times\R^N$ and $\theta_{\xi,\lambda}\le\beta<\theta$. Finally, by definition of~$L^*(\xi)$, one infers that $L^*(\xi)\ge\lambda_0>0$ for all $\xi\in S^{N-1}_+$. This is the desired result and the proof of Theorem~\ref{th:periodic_pb} is thereby complete.
\end{proof}

\hfill\break\indent
As was outlined in Remark~\ref{rem:periodic_pb}, it is straightforward to deal with the asymptotic behavior of~$v_{\xi,\lambda}$ when the hypothesis~\eqref{eqn:cond_v0} is not satisfied. For any $(1,...,1)$-periodic function $v_0\in L^{\infty}(\R^N,[0,1])$, it is therefore consistent to extend the definition of $L^* (\xi)$ as follows:
\begin{equation}\label{eqn:extend}
L^* (\xi) = \left\{ 
\begin{array}{lll}
+\infty & \ds\mbox{if } \mathop{\mathrm{ess} \sup}_{\R^N} v_0 \leq \theta & \!\!\hbox{and }\ds0\le\mathop{\mathrm{ess} \inf}_{\R^N} v_0 <\theta ,\vspace{3pt} \\
0 & \mbox{if }  \ds\theta\le\mathop{\mathrm{ess} \inf}_{\R^N} v_0  , 
\end{array}
\right.
\end{equation}
so that, except when $v_0 \equiv \theta$, the value $L^* (\xi)$ remains the critical period magnitude above which the solution~$v_{\xi,\lambda}$ converges to~$1$ as $t\to+\infty$ uniformly in $\R^N$, and below which it converges to some stationary state in the interval $[0,\theta)$. If $v_0 \equiv \theta$, then $v_{\xi,\lambda} \equiv \theta$ in $(0,+\infty) \times \R^N$ for all $(\xi,\lambda) \in S^{N-1}_+ \times (0,+\infty)$ and, with the convention $L^* (\xi)=0$ in this case, the conclusion of Proposition~\ref{prop:hom} below will still be valid.

We end this section by describing how the problem reduces to an $(N-i)$-dimensional situation as~$\xi$ approaches the boundary of~$S_+^{N-1}$ (up to a permutation of the coordinates, one can assume without loss of generality that $\xi$ approaches $\{0\}^i\times S^{N-i-1}_+$).

\begin{proposition}\label{prop:hom}
Let $v_0\in L^{\infty}(\R^N,[0,1])$ be a $(1,...,1)$-periodic function satisfying~\eqref{eqn:cond_v0}. Fix an integer $1 \leq i \leq N-1$ and define the function
$$v_0^i (x_{i+1},...,x_N) := \int_{[0,1]^i} v_0 (x_1,...x_i,x_{i+1},...,x_{N})\,dx_1\, ...\, \,dx_i,$$  
which belongs to $L^{\infty}(\R^{N-i},[0,1])$ and is $(1,...,1)$-periodic in $\R^{N-i}$. Applying Theorem~$\ref{th:periodic_pb}$ together with the definition~\eqref{eqn:extend}, we define the mappings $L^* : S_+^{N-1} \rightarrow[0,+\infty)$ and $L^*_i :S_+^{N-i-1} \rightarrow[0,+\infty]$ associated respectively with the functions $v_0$ and $v_0^i$. Then, for any~$\tilde{\xi}\in S^{N-i-1}_+$, the following convergence holds:
$$\lim_{\xi\in S^{N-1}_+,\, \xi\to(0,...,0,\tilde{\xi})} L^* (\xi) = L^*_i\big(\,\tilde{\xi}\,\big).$$
\end{proposition}

Note that the integral defining $v^i_0(x_{i+1},...,x_N)$ converges for almost every $(x_{i+1},...,x_N)\in\R^{N-i}$ and satisfies $0\le v^i_0(x_{i+1},...,x_N)\le1$. However, the function $v_0^i$ may not satisfy the condition~\eqref{eqn:cond_v0}. In particular, the above proposition shows that the mapping $L^*$ may not be bounded with respect to~$\xi \in S_+^{N-1}$ (as shown in the particular example~\eqref{defv0x12}). Let us now briefly sketch the proof of Proposition~\ref{prop:hom}.\hfill\break

\begin{proof}[Proof of Proposition~$\ref{prop:hom}$]
Fix an integer $1 \leq i \leq N-1$ and $\tilde{\xi}=(\tilde{\xi}_{i+1},...,\tilde{\xi}_N)\in S_+^{N-i-1}$. First of all, if $v^i_0 \equiv \theta$ in $\R^{N-i}$, then $L^*_i (\tilde{\xi})=0$ by~\eqref{eqn:extend}, and $\overline{v_0}=\theta$. Hence, by part~(i) of Theorem~\ref{th:periodic_pb}, $L^* (\xi) = 0$ for all $\xi \in S^{N-1}_+$ and the desired conclusion follows.

Let us consider in the sequel that $v^i_0 \not \equiv \theta$. Assume first that $L^*_i(\tilde{\xi})$ is a real number and consider any real number $\lambda\in(L^*_i (\tilde{\xi}),+\infty)$. Then, denote by $v^i(t,x_{i+1},...,x_N)$ the solution of the $(N-i)$-dimensional periodic problem~\eqref{eqn:RD} in $\R^{N-i}$ with initial condition~$v_0^i(x_{i+1} /(\lambda \tilde{\xi}_{i+1}),...,x_N/(\lambda\tilde{\xi}_N))$. From our choice of~$\lambda$, we know by Theorem~\ref{th:periodic_pb} that~$v^i(t,\cdot)\to1$ uniformly in $\R^{N-i}$ as $t\to +\infty$. In particular, there is $t_0>0$ such that
$$v^i(t_0,\cdot)\ge\frac{1+\theta}{2}\ \hbox{ in }\R^{N-i}.$$
Let us now check that, for any $\xi\in S^{N-1}_+$ close enough to~$(0,...,0,\tilde{\xi})$, the solution $v_{\xi,\lambda}(t,x_1,...,x_N)$ of~\eqref{eqn:RD_per} also converges to $1$ as $t\to+\infty$, uniformly in $\R^N$. Indeed, since
$$v_{\xi,\lambda}(0,x)=v_0\Big(\frac{x_1}{\lambda\,\xi_1},...,\frac{x_N}{\lambda\,\xi_N}\Big),$$
it follows from the definitions of $v_0^i$ and $v^i$ that, as~$S^{N-1}_+\ni\xi=(\xi_1,...,\xi_N)\to(0,...,0,\tilde{\xi}_{i+1},...,\tilde{\xi}_N)$, there holds
$$\max_{(x_1,...,x_N)\in\R^N}\big|v_{\xi,\lambda}(t_0,x_1,...,x_N)-v^i(t_0,x_{i+1},...,x_N)\big|\to0.$$
Therefore, $1\ge v_{\xi,\lambda}(t_0,\cdot)\ge(1+3\theta)/4$ in $\R^N$ for $\|\xi-(0,...,0,\tilde{\xi})\|$ small enough. The maximum principle and assumption~\eqref{ass_f} imply that $v_{\xi,\lambda}(t,\cdot)\to1$ as $t\to+\infty$ and $\theta_{\xi,\lambda}=1$ for $\|\xi-(0,...,0,\tilde{\xi})\|$ small enough. Hence, by Theorem~\ref{th:periodic_pb}, $\lambda>L^*(\xi)$ for $\xi\in S^{N-1}_+$ close enough to $(0,...,0,\tilde{\xi})$, and
$$\limsup_{\xi\in S^{N-1}_+,\,\xi\to (0,...,0,\tilde{\xi})} L^* (\xi) \leq \lambda .$$
Recalling that $\lambda$ could be chosen arbitrary close to $L^*_i (\tilde{\xi})$, one infers that
$$\limsup_{\xi\in S^{N-1}_+,\,\xi\to (0,...,0,\tilde{\xi})} L^* (\xi) \leq L^*_i (\tilde{\xi}).$$
This property is also trivially satisfied if $L^*_i(\tilde{\xi})=+\infty$.

With similar arguments, one gets that $\liminf_{\xi\in S^{N-1}_+,\,\xi\to (0,...,0,\tilde{\xi})} L^* (\xi) \geq\lambda$ for all $0<\lambda< L^*_i (\tilde{\xi})$ (provided that $L^*_i(\tilde{\xi})>0$). Hence, $\liminf_{\xi\in S^{N-1}_+,\,\xi\to (0,...,0,\tilde{\xi})} L^* (\xi) \geq L^*_i (\tilde{\xi})$, whether $L^*_i(\tilde{\xi})$ be $+\infty$, or a positive real number, or $0$. The proof of Proposition~\ref{prop:hom} is thereby complete.\end{proof}


\section{Spreading properties}

We are now in a position to prove Theorems~\ref{th:spreading_ignition} and~\ref{th:critical_ignition}. We refer to the previous section and the proof of Theorem~\ref{th:periodic_pb} for the construction of the maps $\xi \mapsto L^* (\xi)$ and $\lambda \mapsto \theta_{\xi,\lambda}$, and only concern ourselves here with the spreading properties of~$u$.


\subsection{Proof of Theorem~\ref{th:spreading_ignition}}

\begin{proof}[Proof of Theorem~\ref{th:spreading_ignition}] We assume that $u$ is the solution of~\eqref{eqn:RD} where the initial condition~$u_0$ satisfies Assumptions~\ref{ass:ini1} and~\ref{ass:ini2} and $v_0$ satisfies~\eqref{eqn:cond_v0}. For convenience, we will denote $L = \lambda\,\xi = (L_1,...,L_N)$ the period vector in Assumption~\ref{ass:ini2}, with $\lambda=\|L\| $ and $\xi=\xi_L$.

First note that, for any~$\varepsilon >0$, there exists some $T_\varepsilon >0$ such that
$$\|v_{\xi,\lambda} (T_\varepsilon,\cdot) - \theta_{\xi,\lambda} \|_{L^\infty (\R^N)} < \frac{\varepsilon}{2}.$$
Since $f$ is Lipschitz-continuous in $[0,1]$ and both $u$ and $v_{\xi,\lambda}$ range in $[0,1]$, there holds
$$|u(t,x)-v_{\xi,\lambda}(t,x)|\le e^{Mt}w(t,x)\ \hbox{ for all }(t,x)\in(0,+\infty)\times\R^N,$$
where $M$ is the Lipschitz norm of $f$ in $[0,1]$ and $w$ is the solution of the heat equation $w_t=\Delta w$ in~$(0,+\infty)\times\R^N$ with initial condition
$$w(0,x)=|u_0(x)-v_{\xi,\lambda}(0,x)|=\Big|u_0(x)-v_0\Big(\frac{x_1}{L_1},...,\frac{x_N}{L_N}\Big)\Big|.$$
Since $w(0,x)\to0$ as $\|x\|\to+\infty$ from Assumption~\ref{ass:ini2}, there holds $w(t,x)\to0$ as $\|x\|\to+\infty$ for every $t>0$. Therefore, there is $R_\varepsilon>0$ such that $\|u(T_{\epsilon},\cdot)- v_{\xi,\lambda}(T_{\epsilon},\cdot) \|_{L^\infty (\R^N \setminus B_{R_\varepsilon} ) } < \varepsilon/2$, where $B_{R_\varepsilon}$ denotes the ball of center 0 and radius~$R_\varepsilon$. Hence,
$$\|u (T_\varepsilon ,\cdot) - \theta_{\xi,\lambda} \|_{L^\infty (\R^N \setminus B_{R_\varepsilon})} < \varepsilon .$$

Let us first consider the case $\lambda > L^* (\xi)$ (that is, $\theta_{\xi,\lambda}=1$). This corresponds to the case $\overline{v_0}\ge\theta$ (for which $L^*(\xi)=0$) and to the case $\overline{v_0}<\theta$ and $\|L\|>L^*(\xi_L)$. By choosing $0<\varepsilon<(1-\theta)/2$ and using the notations of the previous paragraph, one has $u(T_\varepsilon, \cdot)>(1 +\theta)/2$ outside the bounded domain~$B_{R_\varepsilon}$. It then easily follows, using for instance the classical spreading result from~\cite{AW} (in that case, spreading occurs from outside), that $u(t+T_\varepsilon,\cdot)$ uniformly converges to 1 as $t \to +\infty$.

Now, let us consider the case $0< \lambda < L^* (\xi)$, that is, $0<\overline{v_0}<\theta_{\xi,\lambda}<\theta$ from Theorem~\ref{th:periodic_pb}. Then, by choosing $\epsilon>0$ small enough so that $0<\theta_{\xi,\lambda}-\epsilon<\theta_{\xi,\lambda}+\epsilon<\theta$, the function $u$ satisfies
$$ \limsup_{\| x \| \to +\infty} u(T_\varepsilon,x) < \theta_{\xi,\lambda} + \varepsilon < \theta\ \hbox{ and }\ \liminf_{\| x \| \to +\infty} u(T_\varepsilon,x) > \theta_{\xi,\lambda} - \varepsilon >0.$$
Although $u(T_\varepsilon,\cdot)$ decays below the ignition threshold at infinity in all directions, the methods used in~\cite{HS} for front-like initial conditions still apply here (as long as Assumption~\ref{ass:ini1} holds) and will imply that~$u$ spreads to $1$ at least with speed $c^* (\theta_{\xi,\lambda} - \varepsilon)$, and at most with speed $c^* (\theta_{\xi,\lambda} + \varepsilon)$. Since our framework is slightly different from~\cite{HS}, we include a proof here for the sake of completeness. On the one hand, for all $x\in\R^N$,
$$u (T_\varepsilon,x) \leq \chi_{B_{R_{\varepsilon}}} (x) + (\theta_{\xi,\lambda} + \varepsilon ) \chi_{\R^N \setminus B_{R_\varepsilon}} (x)=:U_0(x).$$
The solution $U$ of~\eqref{eqn:RD} with initial condition $U_0$ converges locally uniformly to~1 (because it is above~$u(T_{\epsilon}+\cdot,\cdot)$). Hence, from~\cite{AW}, $U$ spreads from $\theta_{\xi,\lambda}+\varepsilon$ to $1$ with speed $c^* (\theta_{\xi,\lambda}+\varepsilon)$ in the sense of Definition~\ref{def:spread}. Therefore, by the comparison principle, for any $c > c^* (\theta_{\xi,\lambda} + \varepsilon)$,
\begin{equation}\label{eqn:spread_below}
\limsup_{t\to +\infty} \sup_{\|x\| \geq ct} u(t,x) \leq  \theta_{\xi,\lambda} + \varepsilon.
\end{equation}
On the other hand, as $\liminf_{\|x\| \rightarrow +\infty} u (T_\varepsilon, x) > \theta_{\xi,\lambda} - \varepsilon$ and since $u (t+T_\varepsilon,x)$ is a supersolution of the heat equation with initial condition $u(T_\varepsilon,\cdot)$, we get the existence of $T'_\varepsilon > T_\varepsilon$ such that
$$\inf_{\R^N} u (T'_\varepsilon,\cdot)> \theta_{\xi,\lambda} - \varepsilon.$$
Let us fix $R>0$ large enough so that the solution of~\eqref{eqn:RD} with initial condition $((1+\theta)/2)\chi_{B_R}$ converges to $1$ locally uniformly in $\R^N$ as~$t\to+\infty$, which is possible by~\cite{AW}. Since $u(t,\cdot)\to1$ as~$t\to+\infty$ locally uniformly in $\R^N$ and since the constant $\theta_{\xi,\lambda}-\epsilon\in(0,\theta)$ is a subsolution of~\eqref{eqn:RD}, there is $T''_{\epsilon}>T'_{\epsilon}$ such that, for all $x\in\R^N$,
$$u (T''_\varepsilon,x) \geq \frac{1+\theta}{2} \chi_{B_R}(x) + (\theta_{\xi,\lambda} - \varepsilon) \chi_{\R^N \setminus B_R }(x)=:\underline{u}_0(x).$$
By the maximum principle and the choice of $R>0$, the solution $\underline{u}$ of~\eqref{eqn:RD} with initial condition $\underline{u}_0$ still converges to $1$ as $t\to+\infty$ locally uniformly in $\R^N$. It is then known by~\cite{AW} that this solution $\underline{u}$ spreads from $\theta_{\xi,\lambda}-\varepsilon$ to~$1$ with speed $c^* (\theta_{\xi,\lambda} - \varepsilon)$. It follows from the comparison principle that for any~$c\in[0,c^* (\theta_{\xi,\lambda} - \varepsilon))$,
\begin{equation}\label{eqn:spread_above}
\lim_{t\to +\infty} \sup_{\|x\| \leq ct} | u(t,x) -1 | =0,
\end{equation}
and for any $c > c^* (\theta_{\xi,\lambda} - \varepsilon)$,
\begin{equation}\label{eqn:spread_below1}
\liminf_{t\to +\infty} \inf_{\|x\| \geq ct} u(t,x) \geq \theta_{\xi,\lambda} - \varepsilon.
\end{equation}
Putting~\eqref{eqn:spread_below},~\eqref{eqn:spread_above} and~\eqref{eqn:spread_below1} together and passing to the limit as $\varepsilon \to 0$ (remember that~$\alpha \in [0,\theta] \mapsto c^* (\alpha)$ is continuous), we conclude that $u$ spreads from~$\theta_{\xi,\lambda}$ to~1 with speed~$c^* (\theta_{\xi,\lambda})$ in the sense of Definition~\ref{def:spread}. The proof of Theorem~\ref{th:spreading_ignition} is thereby complete.
\end{proof}


\subsection{Proof of Theorem~\ref{th:critical_ignition}}\label{sec:proof_critical}

We now turn to the proof of Theorem~\ref{th:critical_ignition}. Namely, we assume that $\overline{v}_0<\theta$ and we consider the critical case $\lambda = \|L\| = L^* (\xi)$ with $\xi=\xi_L$ and an initial condition $u_0\in L^{\infty}(\R^N,[0,1])$ satisfying Assumptions~\ref{ass:ini1} and~\ref{ass:ini2}, as well as~\eqref{ass_u0v0}. The same reasoning as in the previous subsection immediately provides a lower estimate of the spreading speed. Namely,~$u$ spreads at least with speed $c^*(\theta)>0$ from~$\theta$ to 1:
$$\left\{\baa{ll}
\forall \, 0\le c < c^* (\theta), & \ds\lim_{t \rightarrow +\infty} \sup_{\|x\| \leq ct} |u(t,x) -1| = 0,\vspace{3pt}\\
\forall \, c > c^* (\theta), & \ds\liminf_{t \rightarrow +\infty} \inf_{\|x\| \geq ct} u(t,x) \geq \theta.\eaa\right.$$
Getting an upper spreading estimate is more intricate and depends on the convergence rate of the initial condition $u_0$ to the scaled periodic function $v_0$ as~$\|x\| \to +\infty$.

Let $c$ be an arbitrary speed such that $c>c^* (\theta)$. There then exists a planar traveling front~$U_c$ connecting~$\theta$ to 1 with speed~$c$, that is, a solution of~\eqref{eqn:front_eq} with $\alpha=\theta$. Moreover, it is known from~\cite{AW} that this traveling front satisfies that $U'_c <0$ and, in the case $f'(\theta^+) >0$, it has the following behavior:
\begin{equation}\label{eqn:asymp_front}
U_c (z)-\theta \sim B\,e^{-\lambda_- z}\ \hbox{ and }\ U'_c (z) \sim -B\,\lambda_-\,e^{-\lambda_- z}\ \hbox{ as }z \rightarrow +\infty,
\end{equation}
where $B$ is a positive constant and
\be\label{deflambda-}
\lambda_- = \frac{c - \sqrt{c^2 - 4 f' (\theta^+)}}{2}>0
\ee
is the smallest of the two roots of $\lambda^2 - c\lambda + f'(\theta^+)=0$. Those are well-defined because~$c^* (\theta) \geq 2\sqrt{f'(\theta^+)}$. On the other hand, when $f'(\theta^+) =0$, then $U_c$ decays to $\theta$ as $z \to +\infty$  more slowly than any exponential and
\begin{equation}\label{eqn:asymp_front_bis}
\lim _{z \to +\infty}\frac{U '_c (z)}{U_c (z)-\theta} = 0.
\end{equation}
In both cases, the exponential rate $\lambda_0$ of the convergence of $u_0(x)$ to $v_0(x_1/L_1,...,x_N/L_N)$ as~$\|x \| \to +\infty$ (actually of the upper bound of the difference, see the assumptions~\eqref{ass_u0v0} and~\eqref{deflambda*} in Theorem~\ref{th:critical_ignition}) is faster than the convergence rate of $U_c$ to $\theta$. Namely, $\lambda_0\ge\lambda^*>\lambda_->0$ if $f'(\theta^+)>0$, and $\lambda_0>0$ if~$f'(\theta^+)=0$.

The key-point of the argument will be devoted to the construction of appropriate supersolutions. Because of the oscillations of $u$ around $\theta$ as $\|x\| \to +\infty$, it is in general not possible to shift $U_c$ in any direction in order to put it above $u(t,\cdot)$, even for large time $t$. Before dealing with this issue, we slightly modify $U_c$ into a supersolution which is equal to $1$ behind the propagation (see also~\cite{roquejoffre1,roquejoffre2}).

\begin{claim}\label{claim:super_sol}
Let $c>c^*(\theta)$ be fixed and let us call $c^*=c^*(\theta)$ for the sake of simplicity. Then, there exist~$\nu >0$, $\eta_0>0$ and an increasing and bounded function $[0,+\infty)\ni t\mapsto\zeta(t)$, such that, for any~$\eta\in(0,\eta_0]$ and any unit vector $e\in S^{N-1}$, the function
\be\label{defUc}
(t,x)\mapsto\overline{U}_c (t,x;\eta,e) = \min \big\{ 1,U_c (x\cdot e  - ct  - \eta\,\zeta(t)) + \eta\, e^{-c^*(x\cdot e -ct)/2 - \nu t} \big\},
\ee
is a supersolution of~\eqref{eqn:RD} in $[0,+\infty)\times\R^N$.
\end{claim}

\begin{remark}{\rm 
There are some simpler ways to construct supersolutions which would still be suitable for our purpose. For instance, one may check that the function
$$\min\big\{1,U_c (x\cdot e - c't) + \eta\,e^{-c^*(x\cdot e - c' t)/2}\big\},$$
with $c^* (\theta) < c < c' $ and the right choice of parameters, also defines a supersolution of~\eqref{eqn:RD}. However, our choice~\eqref{defUc} is motivated by the fact that we will use in Section~$\ref{sec:CV_profile}$ a family of supersolutions which is similar to~\eqref{defUc}.}
\end{remark}

\begin{proof}[Proof of Claim~$\ref{claim:super_sol}$]
We first choose
\be\label{defnu}
\nu = \displaystyle -\frac{c^{*2} - 2 c c^* + 4 f'(\theta^+)}{8}
\ee
and observe that $\nu >0 $ since $c > c^*\geq 2 \sqrt{f'(\theta^+)}$ and $c^*>0$. We also fix $\eta_0\in(0,(1-\theta)/2)$ such that
\begin{equation}\label{eqn:right2}
\forall \,\theta \leq s_1 \leq s_2 \leq \theta + 2\,\eta_0 , \quad
f (s_2) - f(s_1) - f' (\theta^+ ) (s_2 - s_1) \leq \nu (s_2 - s_1),
\end{equation}
and, as $f'(1) <0$,
\begin{equation}\label{eqn:left2}
\forall \,1- \eta_0 \leq s_1 \leq s_2 \leq 1 , \quad f(s_2) \leq f (s_1).
\end{equation}
Next, there exist $z_2<0<z_1$ such that
\be\label{defz12}\left\{\baa{ll}
\forall\,z \geq z_1 , & U_c (z) \leq \theta + \eta_0,\vspace{3pt}\\
\forall\,z \leq z_2 , & U_c (z) \geq 1 - \eta_0.\eaa\right.
\ee
As~$U_c$ is of class $C^1(\R)$ with $U_c'<0$ in $\R$, there exists $\kappa > 0$ such that
$$-U_c ' (z) \geq \kappa\ \mbox{ for any }  z \in [z_2, z_1].$$
We now define
$$\zeta (t) = \left( 1 -  e^{-\nu t} \right) \times \frac{4\nu + c^{*2} - 2cc^* + 4K}{4\kappa \nu} e^{-c^*z_2/2} ,$$
where $K$ is larger than the Lipschitz norm of $f$, as well as large enough so that
$$\zeta(t) \geq 0\ \hbox{ and }\ \zeta' (t)>0\hbox{ for all }t\ge0.$$ 

Until the end of the proof of Claim~\ref{claim:super_sol}, $\eta$ denotes an arbitrary real number in the interval $(0, \eta_0]$ and $e$ any vector of the unit sphere $S^{N-1}$. Let $\overline{U}_c(\cdot,\cdot;\eta,e)$ be defined by~\eqref{defUc}. For all $t\ge0$ and~$x \cdot e \geq ct + \eta\,\zeta(t) + z_1$, we have $x\cdot e-ct\ge0$ and
$$\theta \leq \overline{U}_c (t,x;\eta,e) \leq U_c (x \cdot e-ct - \eta\,\zeta(t)) + \eta\,e^{-c^*(x \cdot e -ct)/2 - \nu t } \leq \theta + 2\,\eta_0<1 $$
by~\eqref{defz12}. Hence, for any such $(t,x)$,~\eqref{defnu} and~\eqref{eqn:right2} yield
\begin{eqnarray*}
\partial_t \overline{U}_c \!-\! \Delta \overline{U}_c \!-\! f(\overline{U}_c) & \!\!\!\geq\!\!\! & - [ U_c '' \!+\!c U_c ' \!+\! f( U_c) ]  \!-\!  \eta\,\zeta ' (t)\,U_c ' \!-\! \eta  \Big(2 \nu \!+\! \frac{{c^*}^2}{4} \!-\! \frac{cc^*}{2} \!+\! f' (\theta^+) \Big) e^{-c^*(x\cdot e \!-\!ct)/2 \!-\! \nu t } \vspace{3pt}\\
& \!\!\!=\!\!\! & -\eta\,\zeta'(t)\,U_c'\vspace{3pt}\\
& \!\!\!\geq\!\!\! & 0,
\end{eqnarray*}
where $\overline{U}_c(\cdot,\cdot;\eta,e)$ and its derivatives are evaluated at $(t,x)$, while $U_c$ and its derivatives are evaluated at $x\cdot e -ct  -\eta\,\zeta(t)$.

Now consider the case where $t\ge0$ and $x \cdot e \leq ct + \eta\,\zeta(t) + z_2$. As $1$ is a trivial solution of~\eqref{eqn:RD}, it is sufficient to deal with the subcase $\overline{U}_c(t,x;\eta,e) < 1$. Then, using~\eqref{eqn:left2} and~\eqref{defz12}, we get that~$f\big(\overline{U}_c(x\cdot e -ct  -\eta\,\zeta(t))\big)\le f\big(U_c(x\cdot e -ct  -\eta\,\zeta(t))\big)$. With~\eqref{defnu} and $c>c^*>0$, it follows that
$$\partial_t \overline{U}_c - \Delta \overline{U}_c - f(\overline{U}_c)\geq- [ U_c '' +c U_c ' + f( U_c) ]  -  \eta\,\zeta ' (t)\,U_c ' - \eta \Big( \nu + \frac{{c^*}^2}{4} -\frac{cc^*}{2} \Big) e^{-c^*(x\cdot e -ct)/2 - \nu t }\geq0.$$

Lastly, for all $t\ge0$ and $x \in \R^N$ such that $z_2 \leq x \cdot e - ct - \eta\,\zeta(t) \leq z_1$ and $\overline{U}_c(t,x;\eta,e) < 1$, there holds
\begin{eqnarray*}
\partial_t \overline{U}_c - \Delta \overline{U}_c - f(\overline{U}_c)  & \geq & - [ U_c '' +c U_c ' + f( U_c) ]  - \eta\,\zeta ' (t)\,U_c ' - \eta \Big( \nu + \frac{{c^*}^2}{4} - \frac{cc^*}{2} + K \Big) e^{-c^*(x\cdot e -ct)/2 - \nu t } \vspace{3pt} \\
& \geq & \eta\,\zeta ' (t)\, \kappa - \eta \Big(\nu + \frac{{c^*}^2}{4}- \frac{cc^*}{2} + K  \Big) e^{-c^*z_2/2 -\nu t},
\end{eqnarray*}
which is nonnegative from our choice of the function $\zeta(t)$.
 
We conclude that, for every $\eta\in(0,\eta_0]$ and $e\in S^{N-1}$, the function $\overline{U}_c(\cdot,\cdot;\eta,e)$ is indeed a supersolution of~\eqref{eqn:RD} in $[0,+\infty)\times\R^N$, with $\nu$, $\eta_0$ and~$\zeta$ satisfying all the required properties of Claim~\ref{claim:super_sol}.
\end{proof}\\

\begin{proof}[End of the proof of Theorem~$\ref{th:critical_ignition}$] The next issue in our construction of a supersolution of~\eqref{eqn:RD} which is above the solution $u$ is to deal with the oscillations of $u$ around $\theta$ ahead of the propagation. Because of these oscillations~$u$ may never lie below any shift of the function~$\overline{U}_c$ defined in Claim~\ref{claim:super_sol} (notice that~$\overline{U}_c(t,x;\eta,e)\to\theta$ as $x\cdot e\to+\infty$ for any $t\ge0$ and $\eta>0$). We recall that $c>c^*=c^*(\theta)$ is fixed, that $\lambda_0>0$ and $\lambda_0\ge\lambda^*$ from~\eqref{deflambda*}, and that $\lambda_->0$ is given by~\eqref{deflambda-} if $f'(\theta^+)>0$. Let us define
\be\label{deflambdac}
\lambda_c = \min \left\{ \frac{c}{2}, \lambda_0 \right\}>0.
\ee
The real number $\lambda_c$ satisfies $0\le\lambda^*\le\lambda_c\le c/2$. Furthermore, $0<\lambda_-<\lambda^*\le\lambda_c\le c/2$ if $f'(\theta^+)>0$. Hence, $\lambda_c^2 - c \lambda_c - f' (\theta^+) < 0$ whether $f'(\theta^+)$ be positive or equal to $0$. Then, there exists $\delta\in(0,\theta)$ such that $\theta+2\delta<1$ and
$$\forall\, (v,w) \in[\theta -\delta, \theta + \delta] \times [0,\delta], \ \ f (v+w) -f(v) -f' (\theta^+)w \leq \frac{| \lambda_c^2 - c \lambda_c - f' (\theta^+) | }{2}\,w.$$
We also recall that $A$ is a positive constant given by the assumption~\eqref{ass_u0v0} of Theorem~\ref{th:critical_ignition}.

\smallbreak
{\it Step 1: auxiliary conditional supersolutions.} Let us now check that for any unit vector~$e\in S^{N-1}$, any $\tau>0$ and any $X \in \R$, the function 
\begin{equation}\label{defoverw}
(t,x)\mapsto\overline{w}(t,x;\tau,X,e) := v_{\xi,\lambda} (t+\tau,x)+ A\,e^{-\lambda_c (x\cdot e -ct -X) }
\end{equation}
is a supersolution of~\eqref{eqn:RD} whenever $(t,x)\in[0,+\infty)\times\R^N$ and
$$(v_{\xi,\lambda} (t+\tau,x), A\,e^{-\lambda_c (x\cdot e -ct -X)}) \in [\theta -\delta, \theta +\delta] \times[0,\delta].$$
In such a case, one has $\overline{w}(t,x;\tau,X,e)\in[\theta-\delta,\theta+2\delta]\subset[0,1]$ and one can indeed compute that:
\begin{equation}\label{eqoverw}\begin{array}{rcl}
\partial_t \overline{w}\!-\!\Delta\overline{w}\!-\!f(\overline{w}) & \!\!\!\!=\!\!\!\! & \left[ \partial_t v_{\xi,\lambda} - \Delta v_{\xi,\lambda} - f(v_{\xi,\lambda})  \right] - \left[ \lambda_c^2 - c\lambda_c + f' (\theta^+) \right] A\,e^{-\lambda_c (x\cdot e -ct-X)} \vspace{3pt}\\
& \!\!\!\!\!\!\!\! & + \left[ f(v_{\xi,\lambda}) + f' (\theta^+) A\,e^{-\lambda_c (x \cdot e -ct-X)} -  f (v_{\xi,\lambda} + A\,e^{-\lambda_c (x\cdot e-ct-X)}) \right]\vspace{3pt}\\
& \!\!\!\!\geq\!\!\!\! & \ds-[ \lambda_c^2\!-\!c\lambda_c\!+\!f' (\theta^+)] A e^{-\lambda_c (x\cdot e -ct-X)} - \frac{| \lambda_c^2\!-\!c \lambda_c\!-\!f' (\theta^+) | }{2}Ae^{-\lambda_c (x \cdot e\!-\!ct\!-\!X)}\vspace{3pt}\\
& \!\!\!\!=\!\!\!\! & \ds\frac{| \lambda_c^2 - c \lambda_c - f' (\theta^+) | }{2}\,A\,e^{-\lambda_c (x \cdot e -ct-X)}\vspace{3pt}\\
& \!\!\!\!\geq\!\!\!\! & 0,
\end{array}
\end{equation}
where $\overline{w}(\cdot,\cdot;\tau,X,e)$ and its derivatives are evaluated at $(t,x)$, while $v_{\xi,\lambda}$ and its derivatives are evaluated at $(t+\tau,x)$. 

\smallbreak
{\it Step 2: a generalized supersolution in $[0,+\infty)\times\R^N$.} Both families of functions~$\overline{U}_c (t,x;\eta,e)$ and~$\overline{w} (t,x;\tau,X,e)$ move with speed~$c$ in the direction $e$ (up to the space-periodic function $v_{\xi,\lambda}$ for~$\overline{w}$), and we shall use those two functions to construct a supersolution of~\eqref{eqn:RD} bounding~$u$ from above in the whole space.

Firstly, up to reducing $\delta>0$ and increasing~$A>0$, recalling the asymptotics~\eqref{eqn:asymp_front} or~\eqref{eqn:asymp_front_bis} of~$U_c$ and noting that $\lambda_c > \lambda_->0$ in case $f'(\theta^+)>0$ (and $\lambda_c>0$ if $f'(\theta^+)=0$), we can assume without loss of generality that the function $\theta + A\,e^{-\lambda_c z}$ intersects~$U_c (z)$ at a unique~$z_0 \in \R$ such that 
\begin{equation}\label{eqn:claim31a}\left\{\baa{l}
U_c (z_0)- \theta = A\,e^{-\lambda_c z_0} = \ds\frac{\delta}{2},\vspace{3pt}\\
\theta + A\,e^{-\lambda_c z}> U_c (z) \mbox{ for all } z<z_0,\vspace{3pt}\\
\theta + A\,e^{-\lambda_c z} < U_c (z) \mbox{ for all } z>z_0 .\eaa\right.
\end{equation}
Roughly speaking, this means that the function $ \theta + A\,e^{-\lambda_c z}$ is steeper than $U_c$. From the previous inequalities, let us fix $\epsilon>0$ such that
\be\label{defepsilon}
0<\epsilon\le\delta,\ U_c\big(z_0-\frac{\ln 2}{\lambda_c}\big)+\epsilon<\theta+2\,A\,e^{-\lambda_cz_0}-\epsilon\hbox{ and }\theta+A\,e^{-\lambda_c(z_0+1)}+\epsilon<U_c(z_0+1)-\epsilon.
\ee

Secondly, recall that $v_{\xi,\lambda} (\tau,\cdot)$ converges to $\theta$ uniformly in $\R^N$ as $\tau \to +\infty$. As a consequence, there is $\tau_0>0$ such that, for all $\tau\ge\tau_0$ and $(t,x)\in[0,+\infty)\times\R^N$,
\be\label{deftau0}
v_{\xi,\lambda} (t+\tau,x) \in [\theta - \varepsilon , \theta + \varepsilon]\,\subset\,[\theta-\delta,\theta+\delta].
\ee
Furthermore, since $U_c(-\infty)=1$ and the function~$\zeta$ in Claim~\ref{claim:super_sol} is bounded, it is also clear that~$\overline{U}_c(t,x;\eta,e)$ converges to $U_c(x\cdot e-ct)$ as the parameter $\eta \to 0$ uniformly with respect to~$(t,x)\in[0,+\infty)\times\R^N$ and $e\in S^{N-1}$. Therefore, there exists~$\eta\in(0,\eta_0]$ such that, for all $e\in S^{N-1}$ and $(t,x)\in[0,+\infty)\times\R^N$,
\be\label{deftau0bis}
\big| \overline{U}_c (t,x;\eta,e) - U_c  (x \cdot e - ct)\big| \leq \varepsilon.
\ee
By using~\eqref{defoverw}-\eqref{deftau0}, one infers that, for any $\tau\ge\tau_0$, $X\in\R$, $e\in S^{N-1}$ and $(t,x)\in[0,+\infty)\times\R^N$,
$$\begin{array}{rcl}
\ds x\cdot e > z_0 +X - \frac{\ln 2}{\lambda_c} + ct & \Longrightarrow & (v_{\xi,\lambda} (t+\tau,x),A\,e^{-\lambda_c (x\cdot e-ct-X)}) \in [\theta -\delta, \theta + \delta] \times [0,\delta]\vspace{3pt}\\
& \Longrightarrow & \partial_t\overline{w}(t,x;\tau,X,e)-\Delta\overline{w}(t,x;\tau,X,e)-f(\overline{w}(t,x;\tau,X,e))\ge0\end{array}$$
and
$$\left\{\baa{lcl}
\ds x\cdot e = z_0 +X - \frac{\ln 2}{\lambda_c} + ct  & \Longrightarrow & \overline{U}_c (t,x-X e;\eta,e)<\overline{w} (t,x;\tau,X,e),\vspace{3pt}\\
x\cdot e= z_0 +X +1 + ct  & \Longrightarrow &  \overline{w} (t,x;\tau,X,e) < \overline{U}_c (t,x-X e;\eta,e).\eaa\right.$$

From the previous paragraph and Claim~\ref{claim:super_sol}, it follows that, for any $\tau\ge\tau_0$, $X\in\R$ and $e\in S^{N-1}$, the following function $\overline{u}(\cdot,\cdot;\tau,X,e)$ defined in $[0,+\infty)\times\R^N$ by
$$\overline{u} (t,x;\tau,X,e) =\left\{
\begin{array}{ll}
\overline{U}_c (t,x-X e;\eta,e) & \!\!\ds\mbox{if } x\cdot e - ct \leq z_0+X - \frac{\ln 2}{\lambda_c}, \vspace{3pt}\\
\min \big\{ \overline{U}_c (t,x-Xe;\eta,e), \overline{w} (t,x;\tau,X,e) \big\} & \!\!\ds\mbox{if } z_0 \!+\!X\!-\!\frac{\ln 2}{\lambda_c}\leq x \cdot e\!-\!ct \leq z_0\!+\!X\!+\!1,\vspace{3pt}\\
\overline{w} (t,x;\tau,X,e) & \!\!\mbox{if } x\cdot e -ct \geq z_0 +X +1,
\end{array}
\right.$$
is a generalized supersolution of equation~\eqref{eqn:RD} in $[0,+\infty)\times\R^N$.

\smallbreak
{\it Step 3: conclusion.} In order to complete the proof, it only remains to check that there exist a large time~$\tau>0$ and a shift $X\in\R$ such that, under our assumptions, $u (\tau,\cdot)$ lies below the supersolution~$\overline{u} (0,\cdot;\tau,X,e)$ for any~$e \in S^{N-1}$.

We first claim that, for any $e\in S^{N-1}$ and $(t,x)\in[0,+\infty)\times\R^N$,
\begin{equation}\label{eqn:exp_decay}
u (t,x) \leq \min \big\{ 1,  v_{\xi,\lambda} (t,x)  + A\,e^{-\lambda_c  (\|x\|- \overline{c} t)} \big\},
\end{equation}
where $\overline{c} = \lambda_c+M/\lambda_c$ and~$M$ denotes the Lipschitz norm of $f$ in $[0,1]$. This comes from the fact that the function $w:=u - v_{\xi,\lambda}$ satisfies the linear equation
$$\partial_t w = \Delta w +\frac{f(u)-f(v_{\xi,\lambda})}{u-v_{\xi,\lambda}}\,w,$$
(with the convention that the quotient is, say, $0$ if the denominator vanishes) of which $A\,e^{-\lambda_c (x \cdot e - \overline{c} t)}$ is a supersolution, for any unit vector~$e\in S^{N-1}$. By assumption~\eqref{ass_u0v0} and definition~\eqref{deflambdac}, the initial condition~$w (0,x) = u_0(x) - v_0(x_1/L_1,...,x_N/L_N)$ lies below this supersolution. Thus, the maximum principle yields
$$u (t,x) - v_{\xi,\lambda} (t,x) \leq A\,e^{- \lambda_c (x \cdot e - \overline{c} t)} $$
for all $e\in S^{N-1}$ and $(t,x)\in[0,+\infty)\times\R^N$. It follows that $u(t,x) \leq v_{\xi,\lambda} (t,x) +  A\,e^{-\lambda_c ( \|x \|- \overline{c} t)}$ for all~$(t,x)\in[0,+\infty)\times\R^N$. Moreover, $u(t,x) \leq 1$ is an immediate consequence of the comparison principle, which ends the proof of the claim~\eqref{eqn:exp_decay}.

Next, it is clear from the definition of $\overline{U}_c$ in Claim~\ref{claim:super_sol} that
$$\inf_{e \in S^{N-1}} \inf_{ x \cdot e \leq  z_0 +1} \overline{U}_c (0,x;\eta,e) > \theta,$$
and that there exists $z_1\in\R$ such that, for all $e\in S^{N-1}$ and $x \cdot e \leq z_1$, there holds $ \overline{U}_c  (0,x;\eta,e) = 1$. Putting those two facts together with~\eqref{eqn:exp_decay} and $\lim_{t\to+\infty}\|v_{\xi,\lambda}(t,\cdot)-\theta\|_{L^{\infty}(\R^N)}=0$, one can find~$\tau>0$ and $X>0$ large enough such that
\begin{equation*}
u(\tau,x) \leq \overline{U}_c (0,x-Xe;\eta,e)\ \hbox{ for all }e \in S^{N-1}\hbox{ and }x \cdot e \leq  z_0 + X +1.
\end{equation*}
Furthermore, up to increasing~$X$, it is also clear from~\eqref{defoverw} and~\eqref{eqn:exp_decay} that
\begin{equation*}
u (\tau,x) \leq \overline{w} (0,x;\tau,X,e)\ \hbox{ for all }e \in S^{N-1}\hbox{ and }x\in\R^N.
\end{equation*}

It now follows from the last two properties that, for all $e \in S^{N-1}$ and $x \in \R^N$,
$$u (\tau,x) \leq \overline{u} (0,x; \tau,X,e).$$
By Step~2 and the maximum principle, we get that $u(t+\tau,x)\le\overline{u}(t,x;\tau,X,e)$ for all~$(t,x)\in[0,+\infty)\times\R^N$ and $e\in S^{N-1}$. Hence, for any speed $c' >c$,
$$\ds\sup_{x \cdot e \geq c' t} u (t+ \tau,x) \leq \sup_{x \cdot e \geq c' t } \overline{u} (t,x;\tau,X,e)\leq\ds\sup_{x \cdot e \geq c' t } \overline{w} (t,x;\tau,X,e)\leq\| v_{\xi,\lambda} (t + \tau,\cdot) \|_{L^{\infty}(\R^N)} + A\,e^{-\lambda_c((c'-c) t -X) },$$
where the second inequality holds for $t>0$ large enough so that $c't\ge c t+ z_0 + X+1$. Recalling that the choice of~$z_0$ and $X$ did not depend on the direction~$e \in S^{N-1}$, we conclude that
$$\limsup_{t \to +\infty} \sup_{ \| x\| \geq c' t } u (t, x) \leq \theta$$
for every $c'>c$. As~$c$ could be chosen arbitrarily close to $c^*=c^* (\theta)$, we have shown that $u$ spreads from $\theta$ to 1 with at most the speed $c^* (\theta)$, which ends the proof of Theorem~\ref{th:critical_ignition}.
\end{proof}

\begin{remark}{\rm 
The arguments used in the proofs of Theorems~\ref{th:spreading_ignition} and~\ref{th:critical_ignition} also prove that, if
$$\lim_{t \to +\infty} \liminf_{ \|x\| \to +\infty} u(t,x) = \lim_{t \to +\infty} \limsup_{\|x\| \to +\infty} u(t,x) = \alpha$$
for some $\alpha \in [0,\theta]$, together with some additional exponential convergence in the case $\alpha = \theta$, then~$u$ spreads at the speed $c^* (\alpha)$ from $\alpha$ to $1$.}
\end{remark}


\section{Convergence of the profile}\label{sec:CV_profile}

In this section, we prove our last main result, namely Theorem~\ref{th:CV_profile}. The proof follows a similar approach as in for instance~\cite{Fife}, as it first consists in `trapping' the large-time behavior of the solution between two shifts of a same traveling front connecting $\theta_{\xi_L,\|L\|}$ to 1. As in~\cite{Fife}, this relies on an appropriate construction of super and subsolutions. 

Here the main difficulty is again to deal with the oscillations of the solution around the invaded state, so that the argument below will share some similarities with the proof of Theorem~\ref{th:critical_ignition}, which was performed in Section~\ref{sec:proof_critical}. Once our solution is `trapped' between two shifts of the traveling front connecting $\theta_{\xi_L,\|L\|}$ to 1, we will be able to conclude thanks to the results of Berestycki and Hamel~\cite{BH2} concerning the so-called generalized transition fronts.


\subsection{Sub and supersolutions}\label{sec:CV_step1}

We place ourselves under the assumptions of Theorem~\ref{th:CV_profile}. From now on, the direction $e \in S^{N-1}$ is fixed. For convenience, we will again denote by $L = \lambda \xi = (L_1,...,L_N)$ the asymptotic period vector of the initial condition $u_0$ in the direction~$e$ (in the sense of Assumption~\ref{ass:front_like}). In particular, as we assume that $\lambda =\|L \| < L^* (\xi)$, the solution $v_{\xi,\lambda}$ of~\eqref{eqn:RD_per} converges uniformly in large-time to $\theta_{\xi,\lambda}\in(0,\theta)$. Recalling that the convergence rate of $v_{\xi,\lambda}$ to $\theta_{\xi,\lambda}$ is exponential (see Theorem~\ref{th:periodic_pb}), we also introduce some~$C_1>0$ and $\mu_1 >0$ such that, for all $t \geq 0$ and $x \in \R^N$, 
\begin{equation}\label{eqn:CV_per}
|v_{\xi, \lambda} (t,x) - \theta_{\xi , \lambda } | \leq C_1 e^{-\mu_1 t}.
\end{equation}
Next, we note that in the assumptions of Theorem~\ref{th:CV_profile}, we can increase $A>0$ and decrease $\lambda_0>0$ without loss of generality so that~\eqref{u0v0} still holds, together with
\begin{equation}\label{eqn:lambda}
0 < \lambda_0 < c^* (0),
\end{equation}
where we recall that $c^*(0)$ is the unique speed of a traveling front connecting $0$ to $1$. This implies that~$\lambda '\, ^2  - c\lambda ' <0$ for all $c \geq c^* (0)$ and $0 < \lambda ' \leq \lambda_0$. In particular, the exponential function $e^{-\lambda ' (x-ct)}$ is then a supersolution of the heat equation. As the heat equation is also satisfied by the solution around the invaded state, such exponentials will be crucial below.

Lastly, before we begin our construction of appropriate super and subsolutions, we recall that for every~$\alpha\in[0,\theta)$, $U_\alpha$ denotes the unique (up to some shift) traveling front connecting $\alpha$ to $1$ and solving~\eqref{eqn:front_eq}. It is known from~\cite{AW} that $U'_\alpha < 0$ in $\R$ and that
\begin{equation}\label{eqn:exp_behav43}
U_\alpha (z)-\alpha \sim B\,e^{-c^* (\alpha) z} , \ U ' _\alpha (z) \sim - B\,c^* (\alpha)\,e^{-c^* (\alpha) z},
\end{equation}
as $z \to +\infty$, for some constant $B>0$.

\subsubsection*{Sub and supersolutions $\underline{U}$ and $\overline{U}$ which are close to $U_{\theta_{\xi,\lambda}}(x\cdot e-c^*(\theta_{\xi,\lambda})t)$ for large $x\cdot e$}

The following claim is a standard construction of sub and supersolutions which converge to (a shift of)~$U_\alpha(x\cdot e-c^*(\alpha)t)$ as $t \to +\infty$. Again, the novelty of our proof will be to force our sub and supersolutions to oscillate around the invaded state, in order to be able to compare it with the solution~$u$.

\begin{claim}\label{claim:super_solb}
Let $\alpha \in (0,\theta)$ be fixed. There exist $\nu >0$, $\eta_0 >0$ and an increasing and bounded function~$[0,+\infty)\ni t\mapsto\zeta(t)$ such that, for any $\eta\in(0,\eta_0]$, the functions
$$\baa{rcl}
(t,x)\mapsto\overline{U}_\alpha (t,x;\eta) & \!\!\!=\!\!\! & \min \{ 1, U_\alpha (x\cdot e  - c^* (\alpha ) t  - \eta\,\zeta(t)) + \eta\,e^{- \lambda_0 (x\cdot e -c^* (\alpha) t )/2 - \nu t} \},\vspace{3pt}\\
(t,x)\mapsto\underline{U}_\alpha (t,x;\eta) & \!\!\!=\!\!\! & \max\big\{0,U_\alpha (x\cdot e  - c^* (\alpha) t  + \eta\,\zeta(t)) - \eta\,e^{-\nu t} \min \{ 1, e^{- \lambda_0(x\cdot e -c^* (\alpha)t + \eta\,\zeta(+\infty))/2}\}\big\} ,\eaa$$
are, respectively, a supersolution and a subsolution of~\eqref{eqn:RD} in $[0,+\infty)\times\R^N$.
\end{claim}

\begin{remark}{\rm
Here and unlike in Claim~\ref{claim:super_sol}, the direction $e$ is fixed so that we omit it in our notations. However, one may check that all constants, throughout this section, could be chosen independently of~$e \in S^{N-1}$.}
\end{remark}

\begin{proof} The proof is similar to that of Claim~\ref{claim:super_sol} above. We choose
\be\label{defnu2}
\nu = \min \left\{ -\frac{f'(1)}{2}, \frac{ -\lambda_0^2 + 2 c^* (\alpha) \lambda_0 }{4} \right\} >0,
\ee
and $\eta_0\in(0,1)$ such that
$$\left\{\baa{l}
0 \leq \alpha - 2\,\eta_0 < \alpha + 2\,\eta_0 < \theta<1-2\,\eta_0<1,\vspace{3pt}\\
\ds\forall \ 1-  2\,\eta_0  \leq s_1 \leq s_2 \leq 1 , \quad f(s_2) \leq f (s_1) + \frac{f'(1)}{2} (s_2 - s_1) \leq f(s_1).\eaa\right.$$
We then let $z_2<0<z_1$ be such that
$$\left\{\baa{ll}
\forall\, z \geq z_1 , & U_\alpha (z) \leq \alpha + \eta_0,\vspace{3pt}\\
\forall\, z \leq z_2 , & U_\alpha (z) \geq 1 - \eta_0 ,\eaa\right.$$
and define 
$$ \kappa := \inf_{[z_2 , z_1]} -U_\alpha ' (z) >0 .$$
We now set 
$$\zeta(t) = \left( 1 -  e^{-\nu t} \right) \times \frac{ K + \nu }{\kappa \nu}\, e^{-\lambda_0z_2/2}$$
for $t\ge0$, where $K$ is a positive constant larger than the Lipschitz norm of $f$. The function $\zeta$ is nonnegative, increasing and bounded in $[0,+\infty)$. 

Let us now check that, with the choice of parameters above, the functions~$\overline{U}_\alpha(\cdot,\cdot;\eta)$ and $\underline{U}_\alpha(\cdot,\cdot;\eta)$ are respectively a super and a subsolution of~\eqref{eqn:RD}, for any fixed $\eta\in(0,\eta_0]$.

First, for all $t \ge0$ and $x \cdot e \geq c^* (\alpha)\, t + \eta\,\zeta(t) + z_1$, we have $x \cdot e-c^* (\alpha)\, t\ge0$ and
$$0<\overline{U}_\alpha (t,x;\eta) \leq U_\alpha (x \cdot e -c^* (\alpha)\, t - \eta\, \zeta(t)) + \eta\, e^{-\lambda_0(x \cdot e -c^* (\alpha) t )/2 - \nu t } \leq\alpha+2\,\eta_0\le\theta.$$
Thus, $\overline{U}_\alpha (t,x;\eta)=U_{\alpha}(x \cdot e -c^* (\alpha)\, t - \eta\, \zeta(t))+ \eta\, e^{-\lambda_0(x \cdot e -c^* (\alpha) t )/2-\nu t}$ and
\begin{eqnarray*}
\partial_t \overline{U}_\alpha - \Delta \overline{U}_\alpha - f(\overline{U}_\alpha) & \!\!\!=\!\!\! & \partial_t \overline{U}_\alpha - \Delta \overline{U}_\alpha  \vspace{3pt}\\
& \!\!\!\geq\!\!\! & - [ U_\alpha '' +c^* (\alpha) U_\alpha '  ]  -  \eta\,\zeta ' (t) U_\alpha'-\eta \left( \nu + \frac{\lambda_0^2}{4} - \frac{c^* (\alpha) \lambda_0}{2} \right) e^{-\lambda_0(x\cdot e -c^* (\alpha) t )/2 - \nu t } \vspace{3pt}\\
& \!\!\!=\!\!\! & -  \eta\,\zeta ' (t) U_\alpha'-\eta \left( \nu + \frac{\lambda_0^2}{4} - \frac{c^* (\alpha) \lambda_0}{2} \right) e^{-\lambda_0(x\cdot e -c^* (\alpha) t )/2 - \nu t }\vspace{3pt}\\
& \!\!\!\geq\!\!\! & 0,
\end{eqnarray*}
where $\overline{U}_\alpha(\cdot,\cdot;\eta)$ and its derivatives are evaluated at $(t,x)$, while $U_\alpha$ and its derivatives are evaluated at $x \cdot e - c^* (\alpha)t - \eta\,\zeta (t)$. In the above inequality, we used the fact that $\zeta'(t)\ge0$ and $U'_{\alpha}\le0$, as well as~\eqref{ass_f},~\eqref{eqn:front_eq} and~\eqref{defnu2}. Similarly, one can also check that, if $x \cdot e \geq c^* (\alpha ) t - \eta\,\zeta(t) + z_1$ and~$\underline{U}_{\alpha}(t,x;\eta)>0$, then 
$$\baa{rcl}
0<\underline{U}_\alpha (t,x;\eta) & = & U_\alpha (x \cdot e - c^* (\alpha)\,t + \eta\,\zeta(t) ) - \eta\,e^{-\lambda_0( x \cdot e - c^* (\alpha ) t + \eta\,\zeta(+\infty))/2-\nu t}\vspace{3pt}\\
& \leq & U_\alpha (x \cdot e - c^* (\alpha)\,t + \eta\,\zeta(t) )  \leq\alpha+\eta_0\le\theta,\eaa$$
and it satisfies 
\begin{eqnarray*}
\partial_t \underline{U}_\alpha\!\!-\!\Delta \underline{U}_\alpha\!\!-\!f(\underline{U}_\alpha) &  \!\!\!\!=\!\!\!\! & \partial_t \underline{U}_\alpha - \Delta \underline{U}_\alpha \vspace{3pt} \\
& \!\!\!\!=\!\!\!\! & -[U_\alpha '' +c^* (\alpha) U_\alpha']\!+\!\eta\zeta ' (t) U_\alpha'\!+\!\eta\Big(\nu\!+\!\frac{\lambda_0^2}{4}\!-\!\frac{c^* (\alpha)\lambda_0}{2}\Big)e^{-\lambda_0( x \cdot e - c^* (\alpha ) t + \eta\,\zeta(+\infty))/2-\nu t} \vspace{3pt}\\
& \!\!\!\!\leq\!\!\!\! & 0,
\end{eqnarray*}
where $\underline{U}_\alpha(\cdot,\cdot;\eta)$ and its derivatives are evaluated at $(t,x)$, while $U_\alpha$ and its derivatives are evaluated here at $x \cdot e - c^* (\alpha)t+\eta\,\zeta (t)$.

Now consider the case where $x \cdot e \leq c^* (\alpha) t + \eta\,\zeta(t) + z_2$ and $\overline{U}_\alpha(t,x;\eta)<1$. Then, from our choice of $\eta_0$ and~$z_2$ above, we have $0<1-\eta_0\le U_{\alpha}(x \cdot e-c^* (\alpha) t-\eta\,\zeta(t))< \overline{U}_\alpha(t,x;\eta)<1$ and
\begin{eqnarray*}
\partial_t \overline{U}_\alpha\!\!-\!\Delta \overline{U}_\alpha\!-\!f(\overline{U}_\alpha) & \!\!\!\geq\!\!\! & - [ U_\alpha ''\!+\!c^* (\alpha) U_\alpha'\!+\!f( U_\alpha)]\!-\!\eta\zeta' (t)U_\alpha'\!-\!\eta \Big(\nu\!+\!\frac{\lambda_0^2}{4}\!-\!\frac{c^* (\alpha)\lambda_0}{2}\Big) e^{-\lambda_0(x\cdot e\!-\!c^* (\alpha)t)/2\!-\!\nu t} \vspace{3pt}\\
& \!\!\!\geq\!\!\! & 0.
\end{eqnarray*}
Similarly, if $x \cdot e \leq c^* (\alpha )\,t - \eta\,\zeta(t) + z_2$, then 
$$1>U_{\alpha}( x \cdot e\!-\!c^* (\alpha )t\!-\!\eta\zeta(t))> \underline{U}_\alpha(t,x;\eta) \geq U_\alpha ( x \cdot e\!-\!c^* (\alpha )t\!-\!\eta\zeta(t))\!-\!\eta e^{-\nu t} \geq 1\!-\!2\eta\ge1\!-\!2\eta_0>0.$$
In case $1<e^{-\lambda_0(x\cdot e-c^*(\alpha)t+\zeta(+\infty))/2}$, it follows that
$$\partial_t \underline{U}_\alpha - \Delta \underline{U}_\alpha - f(\underline{U}_\alpha)\leq- [ U_\alpha '' +c^* (\alpha ) U_\alpha ' + f( U_\alpha) ]  + \eta\,\zeta' (t)\, U_\alpha '   + \eta \Big( \nu + \frac{f'(1)}{2}  \Big) e^{- \nu t }\leq0,$$
while, in case $e^{-\lambda_0(x\cdot e-c^*(\alpha)t+\zeta(+\infty))/2}<1$, one has
$$\partial_t \underline{U}_\alpha\!-\!\Delta \underline{U}_\alpha\!-\!f(\underline{U}_\alpha)\leq- [ U_\alpha ''\!+\!c^* (\alpha ) U_\alpha'\!+\!f( U_\alpha) ]\!+\!\eta\zeta' (t)U_\alpha'\!+\!\eta \Big( \nu\!+ \!\frac{\lambda_0^2}{4}\!-\!\frac{c^* (\alpha) \lambda_0}{2}\!+\!\frac{f'(1)}{2}  \Big) e^{- \nu t }\leq0.$$

Lastly, for any $t \geq 0 $ and $ x \in \R^N$ such that $z_2 \leq x \cdot e - c^* (\alpha )\,t - \eta\, \zeta(t) \leq z_1$ and $\overline{U}_\alpha (t,x;\eta)<1$, one has $x \cdot e - c^* (\alpha )\,t\ge z_2$ and
\begin{eqnarray*}
\partial_t \overline{U}_\alpha\!\!-\!\Delta \overline{U}_\alpha\!\!-\!f(\overline{U}_\alpha) & \!\!\!\!\!\geq\!\!\!\!\! & - [ U_\alpha''\!+\!c^* (\alpha) U_\alpha'\!+\!f( U_\alpha)]\!-\!\eta\zeta' (t)U_\alpha'\!-\!\eta\Big( \nu\!+\!\frac{\lambda_0^2}{4}\!-\!\frac{c^* (\alpha) \lambda_0}{2}\!+\!K \Big) e^{-\lambda_0(x\cdot e\!-\!c^* (\alpha) t)/2 \!-\!\nu t } \vspace{3pt} \\
& \!\!\!\!\!\geq\!\!\!\!\! & \eta\, \zeta' (t)\,  \kappa - \eta\, K\,e^{-\lambda_0z_2/2} e^{-\nu t},
\end{eqnarray*}
which is nonnegative from our choice of the function $\zeta(t)$. As $1$ is a trivial solution of~\eqref{eqn:RD}, we have then proved that $\overline{U}_\alpha(\cdot,\cdot;\eta)$ is a (generalized) supersolution of~\eqref{eqn:RD} in $[0,+\infty)\times\R^N$.

By a similar computation one can also check that, in the domain $z_2 \leq x \cdot e - ct +\eta\,\zeta(t) \leq z_1$ with~$\underline{U}_{\alpha}(t,x;\eta)>0$, the function $\underline{U}_\alpha$ is a maximum of two subsolutions. By putting this together with the previous inequalities and the fact that $0$ is a trivial solution of~\eqref{eqn:RD}, we conclude that $\underline{U}_{\alpha}(\cdot,\cdot;\eta)$ is a (generalized) subsolution of~\eqref{eqn:RD} in $[0,+\infty)\times\R^N$. The claim is thereby proved.
\end{proof}

\subsubsection*{Sub and supersolutions $\underline{u}$ and $\overline{u}$ which are close to $v_{\xi,\lambda}(t,x)$ for large $x\cdot e$}

In the sequel, we choose $\alpha=\theta_{\xi,\lambda}\in(0,\theta)$ in Claim~\ref{claim:super_solb}. The positive constants $\nu$ and $\eta_0$ and the bounded function $\zeta$ are as in Claim~\ref{claim:super_solb}. Furthermore, we can assume without loss of generality, even if it means redefining $\nu$ in~\eqref{defnu2}, that
$$0<\nu<\frac{\mu_1}{2},$$
where $\mu_1>0$ is as in~\eqref{eqn:CV_per}. We then fix $\eta=\eta_0>0$ in the definition of the functions $\overline{U}_{\theta_{\xi,\lambda}}(\cdot,\cdot;\eta)$ and~$\underline{U}_{\theta_{\xi,\lambda}}(\cdot,\cdot;\eta)$.

Due to its oscillations around the steady state $\theta_{\xi, \lambda}$, the solution~$u$ does not lie below (respectively above) any shift of the supersolution $\overline{U}_{\theta_{\xi,\lambda}}(\cdot,\cdot;\eta)$ (respectively subsolution $\underline{U}_{\theta_{\xi,\lambda}}(\cdot,\cdot;\eta)$) defined in Claim~\ref{claim:super_solb}. We therefore introduce some functions, defined for $\tau\ge0$ and $X\in\R$ as follows:
$$\left.
\begin{array}{l}
(t,x)\mapsto\overline{w}(t,x;\tau,X) := v_{\xi,\lambda} (t+\tau,x)+ A\,e^{-\lambda_0 (x\cdot e -c^* (\theta_{\xi,\lambda})t -X) },\vspace{3pt}\\
(t,x)\mapsto\underline{w}(t,x;\tau,X) := v_{\xi,\lambda} (t+\tau,x)- A\,e^{-\lambda_0 (x\cdot e -c^* (\theta_{\xi,\lambda}) t -X) }.
\end{array} \right.$$
Since $\lambda_0\,c^*(\theta_{\xi,\lambda})-\lambda_0^2>0$, one infers that, for any $\tau \geq 0$ and $X \in \R$, the function $\overline{w}(\cdot,\cdot;\tau,X)$ is a supersolution of~\eqref{eqn:RD} as soon as it lies below~$\theta$, while~$\underline{w}(\cdot,\cdot;\tau,X)$ is a subsolution of~\eqref{eqn:RD} as soon as~$v_{\xi,\lambda} (t+ \tau , x) \leq \theta$. 

Note that unlike in the proof of Theorem~\ref{th:critical_ignition} in Section~\ref{sec:proof_critical}, the traveling front $U_{\theta_{\xi,\lambda}}(z)$ decays to~$\theta_{\xi,\lambda}$ as $z \to +\infty$ faster than the exponential function $\theta_{\xi,\lambda} + e^{-\lambda_0 z}$: recall indeed~\eqref{eqn:lambda} and~\eqref{eqn:exp_behav43}. However, in the definition of $\overline{U}_{\theta_{\xi,\lambda}}$ in Claim~\ref{claim:super_solb} above, the dominating term as $x \cdot e  \to +\infty$ is the exponential~$\eta\,e^{- \lambda_0(x\cdot e -c^* (\theta_{\xi,\lambda}) t )/2 - \nu t}$. More precisely, it is straightforward to check that, for any $\varepsilon >0$, there exists $z_\varepsilon \in \R$ such that, for all $t \geq 0$ and $x \in \R^N$:
\begin{equation}\label{eqn:111}
x \cdot e \geq   c^* (\theta_{\xi,\lambda})\,t  + \frac{\nu\,t}{c^* (\theta_{\xi,\lambda}) - \lambda_0/2} + z_\varepsilon\ \Longrightarrow \  \left| \frac{\overline{U}_{\theta_{\xi,\lambda}}(t,x;\eta) - \theta_{\xi,\lambda}}{\eta\,e^{-\lambda_0( x \cdot e - c^* (\theta_{\xi,\lambda}) t )/2 - \nu t}} - 1\right| \leq  \varepsilon.
\end{equation}
Note that the fact that $z_\varepsilon$ does not depend on $t \geq 0$ is a consequence of the boundedness of $\zeta$ in $[0,+\infty)$. Note also that $z_{\varepsilon}$ is independent of $A$.

This leads us to look at the intersection of the functions $\theta_{\xi,\lambda} +\eta\,e^{-\lambda_0( x \cdot e - c^* (\theta_{\xi,\lambda}) t  )/2 - \nu t}$ and $\theta_{\xi,\lambda} + A\,e^{-\lambda_0 (x \cdot e - c^* (\theta_{\xi,\lambda})t )}$. Clearly, for any $t \geq 0$, these functions intersect exactly along the hyperplane
$$x \cdot e =  c^* (\theta_{\xi,\lambda})\,t + \frac{2 \nu }{\lambda_0} t + \frac{2}{\lambda_0} \ln \frac{A}{\eta} .$$
We then recall~\eqref{eqn:CV_per}, namely
$$|v_{\xi,\lambda} (t,x) - \theta_{\xi,\lambda} | \leq C_1 e^{-\mu_1 t}$$
for all $t\ge0$, with $C_1>0$, $\mu_1>0$, $\theta_{\xi,\lambda}\in(0,\theta)$, and that $2\nu <\mu_1$. Therefore, there exist $\varepsilon>0$ small enough (only depending on $\lambda_0$, namely, $(1-\varepsilon)\,e^{-\lambda_0/2}>e^{-\lambda_0}$ and $(1+\varepsilon)\,e^{\lambda_0/2}<e^{\lambda_0}$) and $\tau_0>0$ large enough (depending on $\varepsilon$, $\lambda_0$, $\eta$, $A$, $\nu$, $C_1$, $\mu_1$ and $\theta_{\xi,\lambda}$) so that, for all $\tau\ge\tau_0$ and $t \geq 0$,
\begin{equation}\label{eqn:222}\left\{\baa{rcl}
\ds (1 - \varepsilon)\,\eta\, e^{-\frac{\lambda_0}{2} \big( \frac{2 \nu}{\lambda_0} t + \frac{2}{\lambda_0} \ln \frac{A}{\eta}  + 1 \big) - \nu t} & > & \ds C_1 e^{-\mu_1 (t+ \tau)} + A\,e^{-\lambda_0 \big(\frac{2 \nu}{\lambda_0} t +  \frac{2}{\lambda_0} \ln \frac{A}{\eta} +1 \big)},\vspace{3pt}\\
\ds (1 + \varepsilon)\,\eta\,e^{-\frac{\lambda_0}{2} \big( \frac{2 \nu}{\lambda_0} t +\frac{2}{\lambda_0} \ln \frac{A}{\eta}  - 1 \big) - \nu t} & < & -C_1 e^{-\mu_1(t+ \tau)} + A\,e^{-\lambda_0 \big(\frac{2 \nu}{\lambda_0} t +\frac{2}{\lambda_0} \ln \frac{A}{\eta}  - 1 \big)},\vspace{3pt}\\
\|v_{\xi,\lambda}(t+\tau,\cdot)\|_{L^{\infty}(\R^N)} & < & \displaystyle\frac{\theta_{\xi,\lambda}+\theta}{2}\eaa\right.
\end{equation}
Up to increasing the constant $A>0$, we can also assume without loss of generality that 
$$\frac{2}{\lambda_0} \ln \frac{A}{\eta }  - 1 \geq z_\varepsilon\ \hbox{ and }\ \frac{\eta^2\,e^{\lambda_0}}{A}<\frac{\theta-\theta_{\xi,\lambda}}{2}$$
(this is indeed possible since $\varepsilon$, $z_\varepsilon$, $\lambda_0$ and $\eta$ can be chosen independently of $A$). Thus, since
$$\frac{\nu}{c^* (\theta_{\xi,\lambda}) - \lambda_0/2} \leq \frac{\nu}{c^* (0) - \lambda_0/2} \leq \frac{2\nu}{\lambda_0},$$
we also have that
\begin{equation}\label{eqn:333}
c^* (\theta_{\xi,\lambda})\,t + \frac{\nu}{c^* (\theta_{\xi,\lambda}) -\lambda_0/2} t + z_\varepsilon \leq c^* (\theta_{\xi,\lambda})\, t + \frac{2\nu}{\lambda_0} t + \frac{2}{\lambda_0} \ln \frac{A}{\eta}  -1,
\end{equation}
for all $t \geq 0$. Furthermore, for all $\tau\ge\tau_0$, $X\in\R$ and $(t,x)\in[0,+\infty)\times\R^N$,
$$\baa{rcl}
\displaystyle x\cdot e-c^*(\theta_{\xi,\lambda})t\ge\frac{2\nu}{\lambda_0}t+\frac{2}{\lambda_0}\ln\frac{A}{\eta}+X-1 & \Longrightarrow & \displaystyle A\,e^{-\lambda_0(x\cdot e-c^*(\theta_{\xi,\lambda})t-X)}\le\frac{\eta^2\,e^{\lambda_0}}{A}<\frac{\theta-\theta_{\xi,\lambda}}{2}\vspace{3pt}\\
& \Longrightarrow & \overline{w}(t,x;\tau,X)<\theta.\eaa$$

From equations~\eqref{eqn:111},~\eqref{eqn:222} and~\eqref{eqn:333}, we can now conclude that, for any $\tau\ge\tau_0$, $X \in \R$ and $(t,x)\in[0,+\infty)\times\R^N$,
$$\left\{\baa{lcl}
\ds x \cdot e - c^* (\theta_{\xi,\lambda} )\, t =  \frac{2 \nu}{\lambda_0} t +  \frac{2}{\lambda_0} \ln \frac{A}{\eta} + X-1 & \Longrightarrow &  \overline{U}_{\theta_{\xi,\lambda}} (t,x -X e;\eta) < \overline{w} (t,x;\tau,X),\vspace{3pt}\\
\ds x \cdot e - c^* (\theta_{\xi,\lambda} )\, t =  \frac{2 \nu}{\lambda_0} t +  \frac{2}{\lambda_0} \ln \frac{A}{\eta} + X + 1 & \Longrightarrow & \overline{U}_{\theta_{\xi,\lambda}} (t,x -X e;\eta) > \overline{w} (t,x;\tau,X).\eaa\right.$$
Therefore, for any $\tau\ge\tau_0$ and $X\in\R$, the following function $\overline{u}(\cdot,\cdot;\tau,X)$ defines a (generalized) supersolution of~\eqref{eqn:RD} in $[0,+\infty)\times\R^N$:
$$\overline{u} (t,x;\tau,X) =\left\{
\begin{array}{l}
\ds\overline{U}_{\theta_{\xi,\lambda}} (t,x-X e;\eta) \ \mbox{ if } \ x\cdot e  - c^* (\theta_{\xi,\lambda} )\,t \leq \frac{2 \nu}{\lambda_0} t +  \frac{2}{\lambda_0} \ln \frac{A}{\eta}    + X- 1, \vspace{3pt}\\
\ds\min \{ \overline{U}_{\theta_{\xi,\lambda}} (t,x-Xe;\eta), \overline{w} (t,x;\tau,X) \} \vspace{3pt}\\
\ds\qquad\qquad\qquad\qquad\,\mbox{ if } \ \frac{2 \nu}{\lambda_0} t\!+\!\frac{2}{\lambda_0} \ln \frac{A}{\eta}\!+\!X\!-\!1\leq x \cdot e\!-\!c^*(\theta_{\xi,\lambda} )t \leq \frac{2 \nu}{\lambda_0} t\!+\!\frac{2}{\lambda_0} \ln \frac{A}{\eta}\!+\!X\!+\!1,\vspace{3pt}\\
\ds\overline{w} (t,x;\tau,X) \qquad\ \ \ \mbox{ if } \ x\cdot e - c^* (\theta_{\xi,\lambda} )\,t  \geq \frac{2 \nu}{\lambda_0} t +  \frac{2}{\lambda_0} \ln \frac{A}{\eta}   + X + 1 .
\end{array}
\right.$$

By a similar argument, one can check that, even if it means increasing $A>0$ and $\tau_0>0$, the following function $\underline{u}(\cdot,\cdot;\tau,X)$ is a (generalized) subsolution of~\eqref{eqn:RD} in $[0,+\infty)\times\R^N$:
$$\underline{u} (t,x;\tau,X) =\left\{
\begin{array}{l}
\ds\!\!\underline{U}_{\theta_{\xi,\lambda}} (t,x-X e;\eta) \ \ \mbox{ if } \ x\cdot e  - c^* (\theta_{\xi,\lambda} )\,t \leq \frac{2 \nu}{\lambda_0} t +  \frac{2}{\lambda_0} \ln \frac{A}{\eta} + \eta\,\zeta(+\infty)   + X- 1, \vspace{3pt}\\
\ds\!\!\max \{ \underline{U}_{\theta_{\xi,\lambda}} (t,x-Xe;\eta), \underline{w} (t,x;\tau,X) \} \vspace{3pt}\\
\ds\mbox{ if }\frac{2 \nu}{\lambda_0} t\!+\!\frac{2}{\lambda_0} \ln \frac{A}{\eta}\!+\!\eta\,\zeta(+\infty)\!+\!X\!-\!1\!\leq\!x \cdot e\!-\!c^*(\theta_{\xi,\lambda})t\!\leq\!\frac{2 \nu}{\lambda_0}t\!+\!\frac{2}{\lambda_0} \ln \frac{A}{\eta}\!+\!\eta\,\zeta(+\infty)\!+\!X\!+\!1,\vspace{3pt}\\
\ds\!\!\underline{w} (t,x;\tau,X) \qquad\ \ \ \ \mbox{ if } \ x\cdot e - c^* (\theta_{\xi,\lambda} )\,t  \geq \frac{2 \nu}{\lambda_0} t +  \frac{2}{\lambda_0} \ln \frac{A}{\eta} + \eta\,\zeta(+\infty)    + X + 1,
\end{array}
\right.$$
for any $\tau\ge\tau_0$ and $X\in\R$.


\subsection{Concluding the proof of Theorem~\ref{th:CV_profile}}

We are now in a position to complete the proof of Theorem~\ref{th:CV_profile}. Let us note that, in the construction of the super and subsolutions~$\overline{u}(\cdot,\cdot;\tau,X)$ and $\underline{u}(\cdot,\cdot;\tau,X)$ in the previous subsection, we are allowed to choose any large enough~$\tau$, and any real number~$X$.

\smallbreak
{\it Step 1:comparison between $u$, $\overline{u}$ and $\underline{u}$.} We will look here for some large $\tau$ and some shifts $X_-$ and $X_+$ so that 
\begin{equation}\label{eqn:last_claim}
\underline{u} (0,\cdot ;\tau, X_-) \leq u (\tau, \cdot) \leq \overline{u} (0,\cdot ;\tau,X_+).
\end{equation}

First, one can prove proceeding as in Section~\ref{sec:proof_critical} (see the proof of~\eqref{eqn:exp_decay}) that
\begin{equation}\label{eqn:repeat}
|u(t,x) - v_{\xi,\lambda} (t,x) | \leq A\,e^{-\lambda_0 (x \cdot e - \overline{c} t)}
\end{equation}
for all $t \geq 0$ and $x \in \R^N$, where $\overline{c}$ is a positive constant. In particular,
$$u(\tau,x ) \leq \min \{ 1  , \, v_{\xi,\lambda} (\tau,x) + A\,e^{-\lambda_0 ( x \cdot e - \overline{c}\tau)} \}$$
for all $x\in\R^N$. Similarly as in Step~3 of the proof of Theorem~\ref{th:critical_ignition}, it follows that there exist $\tau_1\ge\tau_0$ and a map $X_+:[\tau_1,+\infty)\to\R$ such that, for any $\tau\ge\tau_1$ and $X\ge X_+(\tau)$, one has
$$u (\tau,x) \leq \overline{U}_{\theta_{\xi,\lambda}} (0,x -Xe;\eta)$$
for all $x \cdot e \leq (2/\lambda_0) \ln(A/\eta)+ X+1$. Up to increasing $X_+(\tau)$, it also follows from~\eqref{eqn:repeat} that~$u(\tau,\cdot) \leq \overline{w} (0,\cdot;\tau,X)$ in $\R^N$ for any $\tau\ge\tau_1$ and $X\ge X_+(\tau)$. Hence,
\begin{equation}\label{eqn:fin_a}
u (\tau,\cdot) \leq \overline{u} (0,\cdot;\tau,X)\ \hbox{ in }\R^N
\end{equation}
for any $\tau\ge\tau_1$ and $X\ge X_+(\tau)$.

Let us now establish a lower bound of $u(\tau,\cdot)$ for $\tau$ large enough. Recalling from Assumption~\ref{ass:front_like} that $\liminf_{x\cdot e \to - \infty} u_0 (x) > \theta$, and using for instance the classical results from~\cite{AW}, we know that the solution spreads in the direction~$e$ at least with speed~$c^* (0) $. In particular, since $c^* (0) > \lambda_0$ (see~\eqref{eqn:lambda} above), we get that, for any $Z \in \R$,
\begin{equation}\label{eqn:spread_a}
\lim_{t\to +\infty}\ \inf_{x \cdot e \leq \lambda_0 t + Z} u(t,x) = 1.
\end{equation}
On the other hand, the uniform convergence of $v_{\xi,\lambda}(t,\cdot)$ to~$\theta_{\xi,\lambda} < \theta$ yields the existence of a real number~$T>0$ such that, for all $t \geq T$ and $x \in \R^N$,
$$v_{\xi,\lambda} (t,x) - A\,e^{-\lambda_0 (x\cdot e - \overline{c} T - \lambda_0 (t -T))} \leq \theta .
$$
Using the fact that $f=0$ in~$[0,\theta]$, one infers that the function
$$(t,x)\mapsto\max\big\{ 0 , v_{\xi,\lambda} (t,x) - A\,e^{-\lambda_0 (x \cdot e - \overline{c} T - \lambda_0 (t-T))} \big\}$$
is a subsolution of~\eqref{eqn:RD} in $[T,+\infty)\times\R^N$. Using again~\eqref{eqn:repeat} and the comparison principle, we conclude that for all $t \geq T$ and $x\in\R^N$,
$$u(t,x) \geq  \max\big\{ 0 , v_{\xi,\lambda} (t,x)  - A\,e^{-\lambda_0 (x \cdot e - \overline{c} T - \lambda_0 (t-T))} \big\}.$$
Thus, for all $\tau \geq T$ and $x \in \R^N$,
\begin{equation}\label{eqn:fin_sub1}
u (\tau, x ) \geq \underline{w} (0,x;\tau, \overline{c}T + \lambda_0 (\tau - T)) .
\end{equation}
Let us now note by definition of $\underline{U}_{\theta_{\xi,\lambda}}(\cdot,\cdot;\eta)$ in Claim~\ref{claim:super_solb} that
$$\sup_{x \cdot e \leq (2/\lambda_0) \ln(A/\eta) + \eta\,\zeta (+\infty) +1} \underline{U}_{\theta_{\xi,\lambda}} (0,x;\eta) < 1.$$
Using~\eqref{eqn:spread_a} and letting
$$Z = \frac{2}{\lambda_0} \ln\frac{A}{\eta} + \eta\,\zeta (+\infty) +1 + \overline{c} T - \lambda_0 T,$$
there exists $\tau_2\geq \max(T,\tau_1)$ such that, for any $\tau\ge\tau_2$,
\begin{equation}\label{eqn:fin_sub2}
\inf_{x \cdot e \leq \lambda_0 \tau+  Z} u(\tau,x) \geq \sup_{ x\cdot e \leq  \lambda_0 \tau+  Z } \ \underline{U}_{\theta_{\xi,\lambda}} (0,x - (\lambda_0 \tau + \overline{c} T - \lambda_0 T) e;\eta).
\end{equation}
Putting together~\eqref{eqn:fin_sub1} and~\eqref{eqn:fin_sub2}, we conclude that, for any $\tau\ge\tau_2$,
\begin{equation}\label{eqn:fin_b}
u (\tau,\cdot) \geq \underline{u} (0,\cdot;\tau,X_-)\ \hbox{ in }\R^N
\end{equation}
with $X_- = \lambda_0 \tau + \overline{c} T - \lambda_0 T$.

Combining~\eqref{eqn:fin_a} and~\eqref{eqn:fin_b}, we obtain that the claim~\eqref{eqn:last_claim} holds true for some positive real numbers $\tau>0$ and $X_{\pm}$, which are fixed in the sequel. By the comparison principle, we then get that
\begin{equation}\label{eqn:last_claim_t}
\underline{u} (t-\tau,x;\tau,X_-) \leq u (t,x) \leq \overline{u} (t-\tau,x;\tau,X_+),
\end{equation}
for all $t \geq \tau$ and $x\in \R^N$.

\smallbreak
{\it Step 2: definition of a shift function $m$ and convergence to the planar traveling front.} Define $m(t,y)$ for any $t$ large enough and $y$ orthogonal to~$e$ (that is, $y\in(\R e)^{\perp}$), such that
$$c^* (\theta_{\xi,\lambda})\,t + m(t,y) = \inf  \Big\{ z \in \R \, | \ u(t, z e + y ) \leq \frac{1+\theta}{2} \Big\}.$$
Note that $m$ is a well-defined real valued function in $[t_0,+\infty)\times(\R e)^{\perp}$ (for some $t_0>0$) from~\eqref{eqn:repeat}, $v_{\xi,\lambda}(+\infty,\cdot)=\theta_{\xi,\lambda}<\theta$ and from the spreading property~\eqref{eqn:spread_a} of the solution~$u$. Furthermore, even if it means increasing $t_0>0$, it follows from~\eqref{eqn:last_claim_t} and the definitions of $\overline{u}$ and $\underline{u}$ that the function~$m$ is bounded in $[t_0,+\infty)\times(\R e)^{\perp}$. Let us also extend $m$ by $m=0$ in $[0,t_0)\times(\R e)^{\perp}$ (thus, $m$ is a bounded function in $[0,+\infty)\times(\R e)^{\perp}$).

Let us now consider the family of functions
$$[-s,+\infty)\times\R^N\ni(t,x) \mapsto V_{s,y}(t,x)=u(t+s,x + c^* (\theta_{\xi,\lambda})\,(t+s)\,e+m(s,y)\,e+y),$$
parametrized by $(s,y)\in[0,+\infty)\times(\mathbb{R} e)^\perp$. By standard parabolic estimates, this family $(V_{s,y})$ is relatively compact with respect to the locally uniform topology in $\R\times\R^N$ as $s\to+\infty$. More precisely, there even holds that, for any compact set $K\subset\R\times\R^N$ and any $\beta\in(0,1)$, there is $s_{K,\beta}>0$ such that~$\sup_{(s,y)\in[s_{K,\beta},+\infty)\times(\R e)^{\perp}}\|V_{s,y}\|_{C^{1+\beta/2,2+\beta}_{t,x}(K)}<+\infty$. Let us now prove that these functions $V_{s,y}$ converge to a traveling front as $s \to +\infty$, uniformly with respect to $y \in (\mathbb{R} e)^\perp$. To do so, it is enough to prove that, for any sequence~$(s_k)_{k\in\N}\to +\infty$ and any sequence $(y_k)_{k\in\N}$ in $(\mathbb{R} e)^\perp$, the functions~$V_{s_k,y_k}$ converge to the same traveling front. Consider any such sequences $(s_k)_{k\in\N}$ and $(y_k)_{k\in\N}$. Up to extraction of a subsequence, the functions $V_{s_k,y_k}$ converge in $C^{1,2}_{t,x;loc}(\R\times\R^N)$ to a solution $u_{\infty}$ of~\eqref{eqn:RD} defined in~$\R\times\R^N$. Furthermore, by construction, it is immediate to see that
$$\overline{u} (t,x+ c^* (\theta_{\xi,\lambda})\,t\,e;\tau,X_+)  \to U_{\theta_{\xi,\lambda}} (x \cdot e - X_+ -  \eta\,\zeta(+\infty) ),$$
$$\underline{u} (t,x+ c^* (\theta_{\xi,\lambda})\,t\,e;\tau,X_-)  \to U_{\theta_{\xi,\lambda}} (x \cdot e - X_- +  \eta\,\zeta(+\infty) ),$$
as $t\to+\infty$, where both convergences are uniform with respect to~$x \in \R^N$. Thus, using again~\eqref{eqn:last_claim_t}, we get that
$$U_{\theta_{\xi,\lambda}} (x\cdot e - A_1) \leq u_\infty (t,x) \leq U_{\theta_{\xi,\lambda}} (x \cdot e - A_2)\ \hbox{ for all }(t,x)\in\R\times\R^N,$$
for some real numbers $A_1 \leq  A_2$. In particular, $u_\infty$ is a generalized transition front, an extension of the usual notion of traveling fronts that was introduced by Berestycki and Hamel in~\cite{BH1,BH2}. Roughly speaking, this notion refers to any front-like entire solution such that the width of the front remains uniformly bounded with respect to time. Among other results, it was proved in~\cite{BH2} that, in the homogeneous ignition framework, there are in fact no other generalized transition fronts than the usual traveling fronts. Therefore, applying Theorem~1.14 in~\cite{BH2}, we immediately conclude that~$u_\infty$ is identically equal, up to some shift, to the traveling front profile $U_{\theta_{\xi,\lambda}}(x\cdot e)$, that is, there is $\varsigma\in\R$ such that $u_{\infty}(t,x)=U_{\theta_{\xi,\lambda}}(x\cdot e+\varsigma)$ for all $(t,x)\in\R\times\R^N$. Moreover, we have by construction that~$V_{s_k,y_k}(0,0)=u(s_k,c^*(\theta_{\xi,\lambda})\,s_k\,e+m(s_k,y_k)e+y_k)=(1+\theta)/2$ for $k$ large enough (such that~$s_k\ge t_0$). Hence, $u_\infty (0,0) = (1+\theta)/2$ and the shift $\varsigma$ does not depend on the choice of the subsequence. Without loss of generality, even if it means shifting $U_{\theta_{\xi,\lambda}}$ independently of $(s_k,y_k)_{k\in\N}$, we conclude that
$$u_\infty (t, x) \equiv U_{\theta_{\xi,\lambda}} (x \cdot e )$$
and that the whole family $(V_{s,y})_{(s,y)\in[0,+\infty)\times(\R e)^{\perp}}$ converges to $u_{\infty}$ in $C^{1,2}_{t,x;loc}(\R\times\R^N)$ as $s\to+\infty$ uniformly with respect to $y\in(\R e)^{\perp}$.

In particular,
$$u(s,x+c^*(\theta_{\xi,\lambda})\,s\,e+m(s,x-(x\cdot e)e)\,e)=V_{s,x-(x\cdot e)e}(0,(x\cdot e)\,e)\to U_{\theta_{\xi,\lambda}}(x\cdot e)$$
as $s\to+\infty$, locally uniformly with respect to $x\cdot e\in\R$. Furthermore, from the limits of $U_{\theta_{\xi,\lambda}}$ at~$\pm\infty$ and from the inequalities~\eqref{eqn:last_claim_t}, it is straightforward to see that the convergence also holds for large~$|x\cdot e|$ and thus is uniform with respect to $x$ in $\R^N$. As a consequence, the limit~\eqref{convfront} follows.

Lastly, it also follows from the properties of $u_{\infty}$ that
$$u(s,x+c^*(\theta_{\xi,\lambda})\,s\,e+m(s,x-(x\cdot e)e)\,e+z)=V_{s,x-(x\cdot e)e}(0,(x\cdot e)\,e+z)\to U_{\theta_{\xi,\lambda}}(x\cdot e+z\cdot e)$$
as $s\to+\infty$, locally uniformly with respect to $x\cdot e\in\R$ and $z\in\R^N$. Therefore,
$$\baa{l}
u(s,c^*(\theta_{\xi,\lambda})\,s\,e+m(s,0)\,e)-U_{\theta_{\xi,\lambda}}(m(s,0)-m(s,y))\vspace{3pt}\\
\qquad\qquad=u(s,y+c^*(\theta_{\xi,\lambda})\,s\,e+m(s,y)\,e+(m(s,0)-m(s,y))\,e-y)-U_{\theta_{\xi,\lambda}}(m(s,0)-m(s,y))\to0\eaa$$
as $s\to+\infty$, locally uniformly with respect to $y\in(\R e)^{\perp}$. But $u(s,c^*(\theta_{\xi,\lambda})\,s\,e+m(s,0)\,e)\to U_{\theta_{\xi,\lambda}}(0)$ as $s\to+\infty$. Since $U_{\theta_{\xi,\lambda}}$ is decreasing, one concludes that $m(s,y)-m(s,0)\to0$ as $s\to+\infty$ locally uniformly with respect to $y\in(\R e)^{\perp}$. The proof of Theorem~\ref{th:CV_profile} is thereby complete.


\end{document}